\documentclass[a4paper,oneside,11pt]{article}%
\usepackage{makeidx}
\usepackage[english]{babel}
\usepackage{amsmath}
\usepackage{amsfonts}
\usepackage{amssymb}
\usepackage{stmaryrd}
\usepackage{graphicx}
\usepackage{mathrsfs}
\usepackage{cite}
\usepackage[colorlinks,linkcolor=red,anchorcolor=blue,citecolor=blue,urlcolor=blue]{hyperref}
\usepackage[symbol*,ragged]{footmisc}
\providecommand{\U}[1]{\protect\rule{.1in}{.1in}}
\providecommand{\U}[1]{\protect \rule{.1in}{.1in}}

\newtheorem{theorem}{Theorem}[section]

\newtheorem{corollary}[theorem]{Corollary}

\newtheorem{lemma}[theorem]{Lemma}

\newtheorem{proposition}[theorem]{Proposition}

\newenvironment{proof}[1][Proof]{\noindent \textbf{#1.} }{\  \rule{0.5em}{0.5em}}
\numberwithin{equation}{section}

\begin{document}

\title{Hua's Theorem with the Primes in \\ Piatetski-Shapiro Prime Sets}
%  \author{Min Zhang\footnotemark   \vspace*{-5mm} \\
   %  \small Department of Mathematics, China University of Mining and Technology \vspace*{-5mm} \\
  %  \small  Beijing 100083, P. R. China  }

\author{Jinjiang Li\footnotemark[1]\,\,\,\, \, \& \,\, Min Zhang\footnotemark[2]    \vspace*{-4mm} \\
$\textrm{\small Department of Mathematics, China University of Mining and Technology}^{*\,\dag}$
                    \vspace*{-4mm} \\
     \small  Beijing 100083, P. R. China  }

\footnotetext[2]{Corresponding author. \\
    \quad\,\, \textit{ E-mail addresses}: \href{mailto:jinjiang.li.math@gmail.com}{jinjiang.li.math@gmail.com} (J. Li), \href{mailto:min.zhang.math@gmail.com}{min.zhang.math@gmail.com} (M. Zhang).   }

\date{}
\maketitle

%\begin{center}
  %\small   Department of Mathematics,
  %  China University of Mining and Technology,
  %  Beijing 100083, P. R. China \\
%\end{center}

{\textbf{Abstract}}: In this paper, we study the hybrid problem of Hua's theorem and the Piatetski-Shapiro prime number theorem, and obtain results
in this direction of the nonhomogeneous case $k=3$, which deepen the classical result of Hua.


\section{Introduction and main result}

    In 1937, I. M. Vinogradov~\cite{Vinogradov}~solved the ternary Goldbach problem. He proved that, for sufficiently
    large odd integer~$N,$~there holds
\begin{equation*}
     \sum_{p_1+p_2+p_3=N}1=\frac{1}{2}\mathfrak{S}(N)\frac{N^2}{\log^3N}+O\bigg(\frac{N^2}{\log^4N}\bigg),
\end{equation*}
where~$\mathfrak{S}(N)$~is the singular series
\begin{equation*}
    \mathfrak{S}(N)=\prod_{p|N}\Big(1-\frac{1}{(p-1)^2}\Big)\prod_{p\nmid N}\Big(1+\frac{1}{(p-1)^3}\Big).
\end{equation*}
For any sufficiently large odd integer $N$ and fixed positive integer $k$, let $\mathcal{R}(N,k)$  be the number of
representations of $N$ in the form
\begin{equation*}
  N=p_1+p_2+p_3^k,
\end{equation*}
where $p_1,\,p_2,\,p_3$ are primes. In 1938, L. K. Hua \cite{Hua} generalized the result of Vinogradov and proved that
\begin{equation*}
  \mathcal{R}(N,k)=\frac{k^2}{k+1}\mathfrak{S}(N,k)\frac{N^{1+1/k}}{\log^3N}+O\Big(\frac{N^{1+1/k}}{\log^4N}\Big),
\end{equation*}
where
\begin{equation*}
  \mathfrak{S}(N,k)=\prod_{p}\Big(1+\frac{B_p(N,k)}{(p-1)^3}\Big),
\end{equation*}
$$B_q(N,k)=\sum_{\substack{a=1\\(a,q)=1}}^q C_q(a,k)e\Big(-\frac{aN}{q}\Big),\quad C_q(a,k)=\sum_{\substack{\ell=1\\(\ell,q)=1}}^qe\Big(\frac{a\ell^k}{q}\Big).$$

    In 1986, Wirsing~\cite{Wirsing}, motivated by the earlier work of Erd\H{o}s and Nathanson~\cite{Erdos-Nathanson}~on sums of
 squares, considered the question of whether one could find thin subsets~$\mathcal{S}$~of primes which were still sufficient to obtain all
 sufficiently large odd integers as sums of three of them.~He obtained the very satisfactory answer that there exist such sets S with the
 property that~$\sum\limits_{p\leqslant x,\,p\in\mathcal{S}}1\ll(x\log x)^{1/3}.$~This result was later rediscovered by Ruzsa.~Wirsing's result, which
 is obviously best possible apart from the logarithmic factor, is based on probabilistic considerations and does not lead to a subset of the
 primes which is constructive or recognizable.

    We fix a real number $c$ and consider the number of~$n\leqslant x$~such that the integer part~$[n^{c}]$~is a prime. In the case that $0<c\leqslant1$ every
prime $\leqslant x^c$ occurs in this fashion and it is a simple consequence of the prime number theorem that we have the expected asymptotic formula
\begin{equation}\label{introduction-1}
   \sum_{ \substack{n\leqslant x \\ [n^c]=p } } 1= \big(1+o(1)\big)\frac{x}{c\log x}.
\end{equation}

      We let~$\gamma=1/c$,~so that the set of the Piatetski-Shapiro primes of type $\gamma<1$
\begin{equation*}
  \mathcal{P}_\gamma=\{p:p=[n^{1/\gamma}]~\textrm{for some $n\in\mathbb{N}$}\}
\end{equation*}
is a well-known thin set of prime numbers. Piatetski-Shapiro~\cite{Piatetski-Shapiro}~proved the much more difficult result that the asymptotic
formula (\ref{introduction-1}) still holds in the range~$1<c<12/11.$~Since then, this range for $c$ has been improved by a number of
authors
% In 1983, Heath-Brown~\cite{Heath-Brown-1}~has extended it to~$1<c<755/662$~and Kolesnik~\cite{Kolesnik-2},~in 1985, further to $1<c<39/34.$
%Later, the counting function $\pi_\gamma(x)$ of $\mathcal{P}_\gamma$ was studied by many authors
\cite{Baker-Harman-Rivat,Heath-Brown-1,Jia-1,Jia-2,Kolesnik-1,Kolesnik-2,Kumchev,Leitmann,Liu-Rivat,Rivat}. The best results are
given by \cite{Rivat-Sargos} and \cite{Rivat-Wu}, where it is proved that
\begin{equation*}
      \pi_\gamma(x)\sim \frac{x^\gamma}{\log x}
\end{equation*}
for $2426/2817<\gamma<1$, and
\begin{equation*}
      \pi_\gamma(x)\gg \frac{x^\gamma}{\log x}
\end{equation*}
for $205/243<\gamma<1$.

   In 1992, A. Balog and J. P. Friedlander~\cite{Balog-Friedlander} considered the ternary Goldbach problem with variables restricted to Piatetski-Shapiro primes. They proved that, for $20/21<\gamma\leqslant1$ fixed, any sufficiently large odd integer $N$ can be written as three primes with each prime of the form $[n^{1/\gamma}]$. Rivat \cite{Rivat} extended the range $20/21<\gamma\leqslant1$ to $188/199<\gamma\leqslant1$; Kumchev \cite{Kumchev} extended the range to $50/53<\gamma\leqslant1.$ Jia \cite{Jia-3} used a sieve method to enlarge the range to $15/16<\gamma\leqslant1.$

   In 1998, Zhai \cite{Zhai} considered the hybrid problem of quadratic Waring-Goldbach problem with each prime variable restricted to Piatetski-Shapiro sets.
To be specific, he proved that, for $43/44<\gamma\leqslant1$ fixed, every sufficiently large integer $N$ satisfying $N\equiv 5\pmod {24}$ can be written as five
squares of primes with each prime of the form $[n^{1/\gamma}]$. Later, in 2005, Zhang and Zhai \cite{Zhang-Zhai} improved the result of Zhai \cite{Zhai} and
enlarge the range to $249/256<\gamma\leqslant1.$

  In 2004, Cui \cite{Cui} studied the hybrid problem of Hua's theorem ($k=2$) with each prime variable restricted to Piatetski-Shapiro sets. He proved that, for
  any $104/105<\gamma\leqslant1$ fixed, every sufficiently large odd integer can be written as the sum of two primes and a prime square with all primes of the
  form $[n^{1/\gamma}]$.

   In this paper, we consider the hybrid problem of Hua's theorem ($k=3$) with each prime variable restricted to Piatetski-Shapiro sets and prove the following theorem.

\begin{theorem}\label{Hua's-theorem-k=3}
   Let $0<\gamma_i\leqslant1\,(i=1,2,3),\,0<\delta_j<1\,(j=1,3)$ satisfying
  \begin{equation}
     \left\{
              \begin{array}{l}
                 \displaystyle\frac{\gamma_1+\gamma_2}{2}+\frac{\delta_1}{40}>1, \\
                 \displaystyle\frac{\gamma_1+\gamma_2}{2}+\frac{\delta_3}{3}>1, \\
                 73(1-\gamma_i)+86\delta_1<9 \,\,(i=1,2), \\
                 1714(1-\gamma_3)+1725\delta_3<46.
              \end{array}
     \right.
  \end{equation}
  Then for sufficiently large odd integer $N$, the equation
  \begin{equation*}
     N=p_1+p_2+p_3^3,\quad p_i\in\mathcal{P}_{\gamma_i},\quad i=1,2,3
  \end{equation*}
   is solvable.
\end{theorem}

   From Theorem \ref{Hua's-theorem-k=3}, we know that one may require three summands to be Piatetski-Shapiro
   primes of different type. In particular, by choosing $\gamma_1=\gamma_2=\gamma_3=\gamma$, we obtain

\begin{corollary}\label{corollary-1}
  For any fixed $\frac{2816}{2825}<\gamma\leqslant1,$ every sufficiently large odd integer $N$ can be written as the sum of two primes and a cube of prime
  with all primes of the form $[n^{1/\gamma}]$.
\end{corollary}

    However, the above result is not the best one. Taking $\gamma_1=\gamma_2=\gamma$ as in Corollary \ref{corollary-1}, we can enlarge the range of
    the value of $\gamma_3$ and obtain

\begin{corollary}\label{corollary-2}
  For any fixed $\frac{2816}{2825}<\gamma\leqslant1$ and $\frac{3335}{193682}<\gamma_3\leqslant1$, every sufficiently large odd integer $N$ can be
  written as the sum of two primes of the form $[n^{1/\gamma}]$ and a cube of prime
  with the prime of the form $[n^{1/\gamma_3}]$.
\end{corollary}

  If we take $\gamma_1=\gamma_2=1$, then we can obtain

\begin{corollary}\label{corollary-3}
  For any fixed $\frac{1668}{1714}<\gamma\leqslant1$, every sufficiently large odd integer $N$ can be
  written as the sum of two primes and a cube of prime
  with the last prime of the form $[n^{1/\gamma}]$.
\end{corollary}

\textbf{Notation.} Throughout this paper,~$p,p_1,\cdots$~are
primes; $N$~always denotes a sufficiently large natural number;
$\varepsilon$~always denotes an arbitrary small positive constant, which may not be the same at different occurrences;% ~$p$~always denotes a prime number;~
~$n\sim X$~means~$X<n\leqslant 2X$. We use $[x],\,\{x\}$ and $\|x\|$ to denote the integral part of $x$, the fractional part of $x$ and
 the distance from $x$ to the nearest integer correspondingly. $\Lambda(n)$~denotes von Mangold's function;
~$\mu(n)$~denotes M\"{o}bius function;~$e(x)=e^{2\pi ix};$~$\mathcal{L}=\log N$;~$\psi(x)=x-[x]-\frac{1}{2}$.
 $f(x)\ll g(x)$ means that $f=O(g(x))$; $f(x)\asymp g(x)$ means that $f(x)\ll g(x)\ll f(x)$.

   We also define
\begin{equation*}
  P=N^{1/3},\quad S_1(N,\alpha)=\sum_{p\leqslant N}e(\alpha p),\quad S_3(N,\alpha)=\sum_{p\leqslant P}e(\alpha p^3),
\end{equation*}
\begin{equation*}
  T_1(N,\alpha)=\frac{1}{\gamma}\sum_{\substack{p\leqslant N\\ p\in\mathcal{P}_\gamma}}p^{1-\gamma}e(\alpha p), \quad
  T_3(N,\alpha)=\frac{1}{\gamma}\sum_{\substack{p\leqslant P\\ p\in\mathcal{P}_\gamma}}p^{1-\gamma}e(\alpha p^3),
\end{equation*}
\begin{equation*}
  T_{1,i}(N,\alpha)=\frac{1}{\gamma_i}\sum_{\substack{p\leqslant N\\ p\in\mathcal{P}_{\gamma_i}}}p^{1-\gamma_i}e(\alpha p), \qquad (i=1,2).
\end{equation*}

\section{Preliminary Lemmas }
\begin{lemma}\label{Dirichlet}
   For any real numbers $\alpha$ and $\tau\geqslant1$, there must be integers $a$ and $q$, $(a,q)=1,\,1\leqslant q\leqslant\tau$, such that
   \begin{equation}\label{Diophantine-explicit}
      \alpha=\frac{a}{q}+\lambda,\qquad   \left|\lambda\right|\leqslant\frac{1}{q\tau}.
   \end{equation}
\end{lemma}
\begin{proof}
  See C. D. Pan and C. B. Pan \cite{Pan-Pan}, Lemma 5.19.
\end{proof}

\begin{lemma}\label{S-1-N-alpha}
   Let $\alpha$ be as in Lemma \ref{Dirichlet}. Then
   \begin{equation*}
      S_{1}(N,\alpha)\ll N\mathcal{L}^4\Bigg(\frac{1}{q^{1/2}}+\frac{1}{N^{1/5}}+\frac{q^{1/2}}{N^{1/2}}\Bigg).
   \end{equation*}
\end{lemma}
\begin{proof}
  See Vaughan \cite{Vaughan}, Theorem 3.1.
\end{proof}

\begin{lemma}\label{S-3-N-alpha}
   Let $\alpha$ be as in Lemma \ref{Dirichlet}. Then
   \begin{equation*}
      S_{3}(N,\alpha)\ll N^{1/3+\varepsilon}\Bigg(\frac{1}{q}+\frac{1}{N^{1/6}}+\frac{q}{N}\Bigg)^{1/16}.
   \end{equation*}
\end{lemma}
\begin{proof}
  See Harman \cite{Harman}, Theorem 1.
\end{proof}

\begin{lemma}\label{T_1=S_1+O}
   let $\gamma,\,\delta_1$ satisfy $0<\gamma\leqslant1,\,\delta_1>0$ and
   \begin{equation*}
      73(1-\gamma)+86\delta_1<9.
   \end{equation*}
   Then, uniformly in $\alpha$, we have
   \begin{equation*}
      T_{1}(N,\alpha)=S_{1}(N,\alpha)+O\big(N^{1-\delta_1-\varepsilon}\big)
   \end{equation*}
   where the implied constant depends only on $\gamma$ and $\delta_1$.
\end{lemma}
\begin{proof}
  This is, all in essentials, deduced from the process of the proof of Kumchev \cite{Kumchev-1997}  Theorem 2.
\end{proof}

\begin{lemma}\label{T_1-square-mean-value}
   We have
   \begin{equation*}
     \int_{0}^1 \big|T_{1}(N,\alpha)\big|^2\mathrm{d}\alpha\ll N^{2-\gamma}.
   \end{equation*}
\end{lemma}
\begin{proof}
  See Cui \cite{Cui}, Lemma 6.
\end{proof}

\begin{lemma}\label{Jia-0-Lemma-1}
   Suppose that $f(x):[a,b]\rightarrow \mathbb{R}$ has continuous derivatives of order up
   to $2$ on $[a,b]$, where $1\leqslant a<b\leqslant2a$. Suppose further that
   \begin{equation*}
      0<c_1\lambda_1\leqslant|f'(x)|\leqslant c_2\lambda_1,\quad c_3\lambda_1 a^{-1}\leqslant|f''(x)|\leqslant c_4\lambda_1 a^{-1},\quad x\in[a,b],
   \end{equation*}
   where~$c_j\,(j=1,2,3,4)$~are absolute constants. Then
   \begin{equation}\label{van-de-corput-2}
     \sum_{a<n\leqslant b}e\big(f(n)\big)\ll a^{1/2}\lambda_1^{1/2}+\lambda_1^{-1}.
   \end{equation}
   If $c_2\lambda_1\leqslant1/2,$ then we have
    \begin{equation}\label{van-de-corput-1}
     \sum_{a<n\leqslant b}e\big(f(n)\big)\ll \lambda_1^{-1}.
   \end{equation}
\end{lemma}
\begin{proof}
  See Jia \cite{Jia-0}, Lemma 1.
\end{proof}

\begin{lemma}\label{Heath-Brown-3-Theorem-1}
  Let $k\geqslant3$ be an integer, and suppose that $f(x):[a,b]\rightarrow \mathbb{R}$ has
continuous derivatives of order up to $k$ on $[a,b]$, where $1\leqslant a<b\leqslant2a$. Suppose further that
 \begin{eqnarray*}
   0<\lambda_k\leqslant\big| f^{(k)}(x)\big|\leqslant A\lambda_k,\quad x\in [a,b].
 \end{eqnarray*}
 Then
 \begin{equation*}
    \sum_{a<n\leqslant b}e\big(f(n)\big)\ll_{A,k,\varepsilon} N^{1+\varepsilon}
    \Big( \lambda_{k}^{1/k(k-1)}+N^{-1/k(k-1)}+N^{-2/k(k-1)}\lambda_{k}^{-2/k^2(k-1)} \Big).
 \end{equation*}
\end{lemma}
\begin{proof}
 See Heath-Brown~\cite{Heath-Brown-3}, Theorem 1.
\end{proof}

\begin{lemma}\label{Heath-Brown-1-lemma-1}
  Let~$\mathcal{I}$~be a subinterval of~$(Y,2Y]$~and let $J$ be a positive integer. Then
   \begin{equation*}
       \bigg|\sum_{n\in\mathcal{I}}z_n\bigg|^2  \leqslant\left(1+\frac{Y}{J}\right)
       \sum_{|j|\leqslant J}\left(1-\frac{|j|}{J}\right)\sum_{n,n+j\in\mathcal{I}}z_{n+j}\overline{z_{n}}.
    \end{equation*}
\end{lemma}
\begin{proof}
  See Heath-Brown~\cite{Heath-Brown-1}, Lemma 5.
\end{proof}

\begin{lemma}\label{Heath-Brown-1-lemma-2}
  Suppose that $1/2<\alpha<1,\,H\geqslant1,\,N\geqslant1,\,\Delta>0$. Let $S(H,N,\Delta,\gamma)$ \\
  denote the number of solutions of the inequality
  \begin{equation*}
     \big|h_1n_1^\alpha-h_2n_2^\alpha\big|\leqslant\Delta,\qquad h_1,\,h_2\sim H,\,n_1,\,n_2\sim N.
  \end{equation*}
  Then we have
  \begin{equation*}
     S(H,N,\Delta,\gamma)\ll HN\log2HN+\Delta HN^{2-\alpha}.
  \end{equation*}
\end{lemma}
\begin{proof}
  See the discussion on pp. 256-257 of Heath-Brown~\cite{Heath-Brown-1}.
\end{proof}

\begin{lemma}\label{Heath-Brown-1-lemma-2-1}
   For any $H\geqslant1$, we have
   \begin{equation}\label{psi(theta)}
     \psi(\theta)=-\sum_{0<|h|\leqslant H}\frac{e(\theta h)}{2\pi ih}+O\big(g(\theta,H)\big),
   \end{equation}
   where
   \begin{equation}\label{psi(theta)-fourier}
     g(\theta,H)=\min\bigg(1,\frac{1}{H\|\theta\|}\bigg)=\sum_{h=-\infty}^{\infty}a_he(\theta h)
   \end{equation}
   and
    \begin{equation}\label{psi(theta)-fourier-coefficient-estimate}
    a_h\ll \min\bigg(\frac{\log2H}{H},\frac{1}{|h|},\frac{H}{|h|^2}\bigg).
   \end{equation}
\end{lemma}
\begin{proof}
  See pp. 245 of Heath-Brown~\cite{Heath-Brown-1}.
\end{proof}

\begin{lemma}\label{Heath-Brown-identity}
  Let $z\geqslant1$ and $k\geqslant1$. Then, for any $n\leqslant2z^k$,
  \begin{equation*}
     \Lambda(n)=\sum_{j=1}^k(-1)^{j-1}{k \choose j}\mathop{\sum\cdots\sum}_{\substack{n_1n_2\cdots n_{2j}=n\\ n_{j+1},\cdots,n_{2j}\leqslant z }}
     (\log n_1)\mu(n_{j+1})\cdots\mu(n_{2j}).
  \end{equation*}
\end{lemma}
\begin{proof}
  See pp. 1366-1367 of Heath-Brown~\cite{Heath-Brown-2}.
\end{proof}

\begin{lemma}\label{Graham-Kolesnik}
  Suppose that
  \begin{equation*}
       L(H)=\sum_{i=1}^mA_iH^{a_i}+\sum_{j=1}^nB_jH^{-b_j},
  \end{equation*}
  where~$A_i,\,B_j,\,a_i\,\textrm{and}\,\,b_j$~are positive. Assume that~$H_1\leqslant H_2.$~Then there is some~$\mathscr{H}$\\
  with~$H_1\leqslant\mathscr{H}\leqslant H_2$~and
   \begin{equation*}
      L(\mathscr{H}) \ll \sum_{i=1}^{m}A_iH_1^{a_i}+\sum_{j=1}^{n}B_jH_2^{-b_j}
                         +\sum_{i=1}^m\sum_{j=1}^{n}\big(A_i^{b_j}B_j^{a_i}\big)^{1/(a_i+b_j)}.
   \end{equation*}
   The implied constant depends only on~$m$~and~$n.$~
\end{lemma}
\begin{proof}
 See Graham and Kolesnik~\cite{Graham-Kolesnik},~Lemma~2.4.~
\end{proof}

\begin{lemma}\label{Jia-0-Lemma-3-2}
   Suppose that $f(x)\ll B,\,f'(x)\gg \Delta$ for $x\sim N$. Then we have
   \begin{equation*}
    \sum_{n\sim N}\min\bigg(D,\frac{1}{\|f(n)\|}\bigg)\ll(B+1)\bigg(D+\frac{1}{\Delta}\bigg)\log\bigg(2+\frac{1}{\Delta}\bigg).
   \end{equation*}
\end{lemma}
\begin{proof}
 See Jia~\cite{Jia-0},~Lemma~3.~
\end{proof}

\begin{lemma}\label{Jia-0-Lemma-3-1}
   Suppose $f(x)$ and $g(x)$ are algebraic function in $[a,b]$ and
    \begin{equation*}
     \frac{1}{R}\leqslant\big|f''(x)\big|\ll\frac{1}{R},\qquad \big|f'''(x)\big|\ll\frac{1}{RU}\,\,\,(U\geqslant1),
   \end{equation*}
    \begin{equation*}
     \big|g(x)\big|\leqslant G,\qquad \big|g'(x)\big|\ll U^{-1}G.
   \end{equation*}
   Let $[\alpha,\beta]$ be the image of $[a,b]$ under the mapping $y=f'(x)$. Then we have
   \begin{eqnarray*}
    \sum_{a<n\leqslant b}g(n)e\big(f(n)\big) & = & \sum_{\alpha<u\leqslant\beta}\frac{g(n_u)}{\sqrt{f''(n_u)}}
                                                   e\left(f(n_u)-un_u+\frac{1}{8}\right)  \\
         &  & +O\big( G\log(\beta-\alpha+2)+U^{-1}G(b-a+R)\big) \\
         &  & +O\left(G\min\bigg(\sqrt{R},\frac{1}{\|\alpha\|}\bigg)+G\min\bigg(\sqrt{R},\frac{1}{\|\beta\|}\bigg)\right),
   \end{eqnarray*}
   where $n_u$ is the solution of $f'(n)=u.$
\end{lemma}
\begin{proof}
 See Jia~\cite{Jia-0},~Lemma~5.~
\end{proof}

 For the sum of the form
\begin{equation*}
  \min\bigg(1,\frac{H_1}{H}\bigg)\sum_{h\sim H}\Bigg|\sum_{m\sim M}\sum_{n\sim N}a_mb_ne\big(\alpha m^3n^3+h(mn+u)^\gamma\big)\Bigg|
\end{equation*}
with
\begin{equation*}
  MN\sim x,\,\,\,a_m\ll x^{\varepsilon},\,\,b_n\ll x^{\varepsilon}
\end{equation*}
for every fixed~$\varepsilon,$~it is usually called a ``Type I" sum, denoted by~$S_I(M,N),$~if~$b_n=1$~or~$b_n=\log n;$~otherwise it is called a ``Type II" sum, denoted by~$S_{II}(M,N).$~

\begin{lemma}\label{type-II}
   Suppose that $48(1-\gamma)+48\delta<1,\,|a_m|\ll1,\,|b_n|\ll1,\,MN\asymp x$. Then, for
   \begin{equation}\label{Type-II-M-condition}
      x^{24(1-\gamma)+24\delta+\varepsilon}\ll M\ll x^{\gamma-2\delta-\varepsilon},
   \end{equation}
   we have
   \begin{equation}\label{Type-II-estimation}
     S_{II}(M,N)\ll x^{1-\delta-\varepsilon}.
   \end{equation}
\end{lemma}

\begin{proof}
  Let $Q$ be a positive integer satisfying $1\leqslant Q\leqslant HN\log^{-1}x$. For each $q\,\,(1\leqslant q\leqslant Q)$, define
  \begin{equation*}
     w_q:=\big\{ (n,h):4HN^\gamma(q-1)Q^{-1}<hn^{\gamma}\leqslant 4HN^\gamma q Q^{-1},h\sim H,n\sim N   \big\}.
  \end{equation*}
Then we have
\begin{eqnarray*}
   S & := & \sum_{h\sim H}\Bigg|\sum_{m\sim M}\sum_{n\sim N}a_mb_ne\big(\alpha m^3n^3+h(mn+u)^\gamma\big)\Bigg| \\
   & = &  \sum_{h\sim H}\sum_{m\sim M}\sum_{n\sim N}a_mb_nc_he\big(\alpha m^3n^3+h(mn+u)^\gamma\big)  \\
   & = &  \sum_{m\sim M}a_m \sum_{h\sim H}\sum_{n\sim N}b_nc_he\big(\alpha m^3n^3+h(mn+u)^\gamma\big)  \\
   & = &  \sum_{m\sim M}a_m \sum_{q=1}^Q \sum_{(n,q)\in w_q} b_nc_he\big(\alpha m^3n^3+h(mn+u)^\gamma\big),
\end{eqnarray*}
  where $|c_h|=1,\,h\sim H$. By Cauchy's inequality, we obtain
\begin{eqnarray}\label{|s|^2}
   |S|^2 & \ll & \bigg(\sum_{m\sim M}|a_m|^2\bigg)
                 \left(\sum_{m\sim M}\Bigg|\sum_{q=1}^Q \sum_{(n,q)\in w_q} b_nc_he\big(\alpha m^3n^3+h(mn+u)^\gamma\big)\Bigg|^2 \right)
                      \nonumber   \\
   & \ll & MQ\sum_{m\sim M}\sum_{q=1}^Q\Bigg| \sum_{(n,q)\in w_q} b_nc_he\big(\alpha m^3n^3+h(mn+u)^\gamma\big)\Bigg|^2
                      \nonumber   \\
   & \ll & MQ\sum_{q=1}^Q\sum_{\substack{(n_1,h_1)\in w_q \\(n_2,h_2)\in w_q}}\Bigg|\sum_{m\sim M}e\big(\alpha m^3
               (n_1^3-n_2^3)+h_1(mn_1+u)^\gamma-h_2(mn_2+u)^\gamma\big)\Bigg|
                       \nonumber   \\
   & =: & MQ\sum_{*}\Bigg|\sum_{m\sim M}e\big(f(m)\big)\Bigg|,
\end{eqnarray}
where $f(m)=\alpha m^3(n_1^3-n_2^3)+h_1(mn_1+u)^\gamma-h_2(mn_2+u)^\gamma$. The outer sum runs over all the
quadruples $(h_1,n_1,h_2,n_2)$ with $(h_1,n_1),(h_2,n_2)\in w_q.$

    Let $\lambda=h_1n_1^\gamma-h_2n_2^\gamma.$~Then we have $|\lambda|\leqslant 4HN^{\gamma}Q^{-1}$. It is easy to verify that
\begin{eqnarray*}
   f^{(4)}(m) & = & \gamma(\gamma-1)(\gamma-2)(\gamma-3)\Big(h_1n_1^4(mn_1+u)^{\gamma-4}-h_2n_2^4(mn_2+u)^{\gamma-4}\Big) \\
     & = & \gamma(\gamma-1)(\gamma-2)(\gamma-3)\bigg(\lambda m^{\gamma-4}+O\bigg(\frac{H}{H_1M^4}\bigg)\bigg).
\end{eqnarray*}
Thus, there exists a constant $C(\lambda)>0$ such that $f^{(4)}(m)\asymp |\lambda|M^{\gamma-4}$ for $|\lambda|\geqslant C(\lambda)M^{-\gamma}HH_1^{-1}$.
By Lemma \ref{Heath-Brown-3-Theorem-1} with $k=4$, the estimate of the inner sum in (\ref{|s|^2}) is
\begin{equation}\label{inner-sum-estimate-1}
  \sum_{m\sim M}e\big(f(m)\big)\ll M^{2/3+\gamma/12+\varepsilon}|\lambda|^{1/12}+M^{11/12+\varepsilon}+M^{1-\gamma/24+\varepsilon}|\lambda|^{-1/24}.
\end{equation}
If $|\lambda|<C(\lambda)M^{-\gamma}\,(H\leqslant H_1)$ or $|\lambda|<C(\lambda)M^{-\gamma}HH_1^{-1}\,(H>H_1)$, we use the trivial bound $M$ to estimate the
inner sum.

  By Lemma~\ref{Heath-Brown-1-lemma-2}, the contributions of $M$ to $|S|^2$ are (with $H\leqslant H_1$)
\begin{equation}\label{17}
  \ll M^2Q\big(HN\log2HN+M^{-\gamma}HN^{2-\gamma}\big)\ll M^2QHN\log2HN
\end{equation}
and (with $H> H_1$)
\begin{equation}\label{18}
  \ll M^2Q\big(HN\log2HN+HH_1^{-1}M^{-\gamma}HN^{2-\gamma}\big)\ll M^2QHN\log2HN.
\end{equation}
By noting that $|\lambda|\ll HN^{\gamma}Q^{-1}$, then the contribution of $M^{11/12+\varepsilon}$ to $|S|^2$ is
\begin{eqnarray}\label{18-5}
   & \ll &  MQ\cdot M^{11/12+\varepsilon}\cdot S(H,N,4HN^{\gamma}Q^{-1},\gamma)
                        \nonumber   \\
   & \ll & M^{23/12+\varepsilon}Q\cdot \big(HN\log2HN+HN^{\gamma}Q^{-1}\cdot HN^{2-\gamma}\big)
                         \nonumber   \\
   & \ll & M^{23/12+\varepsilon}H^2N^2
                         \nonumber   \\
   & \ll & M^{-1/12+\varepsilon}H^2x^2.
\end{eqnarray}
Similarly, the contribution of $M^{2/3+\gamma/12+\varepsilon}|\lambda|^{1/12}$ to $|S|^2$ is
\begin{eqnarray}\label{19}
   & \ll &  MQ\cdot M^{2/3+\gamma/12+\varepsilon}\cdot H^{1/12}N^{\gamma/12}Q^{-1/12} \cdot S(H,N,4HN^{\gamma}Q^{-1},\gamma)
                        \nonumber   \\
   & \ll & M^{5/3+\gamma/12+\varepsilon}Q^{11/12}H^{1/12}N^{\gamma/12} \big(HN\log2HN+HN^{\gamma}Q^{-1}\cdot HN^{2-\gamma}\big)
                         \nonumber   \\
   & \ll & M^{-1/3+\varepsilon}Q^{-1/12}H^{25/12}x^{2+\gamma/12}.
\end{eqnarray}
By a splitting argument and Lemma \ref{Heath-Brown-1-lemma-2}, the contributions of $M^{1-\gamma/24+\varepsilon}|\lambda|^{-1/24}$ to $|S|^2$ are
(with $H\leqslant H_1$ and $|\lambda|>C(\lambda)M^{-\gamma}$)
\begin{eqnarray}\label{20}
   & \ll & M^{2-\gamma/24+\varepsilon}Q(\log x)\times \max_{M^{-\gamma}\ll U\ll HN^{\gamma}Q^{-1}}
           \sum_{U<|\lambda|\leqslant2U} |\lambda|^{-1/24}
                          \nonumber   \\
  &  \ll &  M^{2-\gamma/24+\varepsilon}Q(\log x)\times \max_{M^{-\gamma}\ll U\ll HN^{\gamma}Q^{-1}}
            U^{-1/24}\cdot S(H,N,U,\gamma)
                          \nonumber   \\
  & \ll &   M^{2-\gamma/24+\varepsilon}Q(\log x)\times \max_{M^{-\gamma}\ll U\ll HN^{\gamma}Q^{-1}}
            \big(U^{-1/24}HN\log2HN+U^{23/24}HN^{2-\gamma}\big)
                          \nonumber   \\
  & \ll & M^{2-\gamma/24+\varepsilon}Q(\log x) \big(M^{\gamma/24}HN\log2HN+H^{23/24}N^{23\gamma/24}Q^{-23/24}\cdot HN^{2-\gamma}\big)
                          \nonumber   \\
  & \ll & M^{2+\varepsilon}QHN\log^2x+Q^{1/24}H^{47/24}M^{\varepsilon}x^{2-\gamma/24}\log x
\end{eqnarray}
and (with $H> H_1$ and $|\lambda|>C(\lambda)M^{-\gamma}HH_1^{-1}$)
\begin{eqnarray}\label{21}
   & \ll & M^{2-\gamma/24+\varepsilon}Q(\log x)\times \max_{M^{-\gamma}HH_1^{-1}\ll U\ll HN^{\gamma}Q^{-1}}
           \sum_{U<|\lambda|\leqslant2U} |\lambda|^{-1/24}
                          \nonumber   \\
  &  \ll &  M^{2-\gamma/24+\varepsilon}Q(\log x)\times \max_{M^{-\gamma}HH_1^{-1}\ll U\ll HN^{\gamma}Q^{-1}}
            \!\!\!\!\!\!\!\! U^{-1/24}\cdot S(H,N,U,\gamma)
                          \nonumber   \\
  & \ll &   M^{2-\gamma/24+\varepsilon}Q(\log x)\times \max_{M^{-\gamma}HH_1^{-1}\ll U\ll HN^{\gamma}Q^{-1}}
            \!\!\!\!\!\!\!\!\!\!  \big(U^{-1/24}HN\log2HN+U^{23/24}HN^{2-\gamma}\big)
                          \nonumber   \\
  & \ll & M^{2-\gamma/24+\varepsilon}Q(\log x) \big(M^{\gamma/24}(H_1H^{-1})^{1/24}HN\log2HN         \nonumber   \\
   &  &  \qquad   \qquad  \qquad  \qquad  \qquad  \qquad   +H^{23/24}N^{23\gamma/24}Q^{-23/24}\cdot HN^{2-\gamma}\big)
                          \nonumber   \\
  & \ll & M^{2-\gamma/24+\varepsilon}Q(\log x) \big(M^{\gamma/24}HN\log2HN+H^{23/24}N^{23\gamma/24}Q^{-23/24}\cdot HN^{2-\gamma}\big)
                          \nonumber   \\
  & \ll & M^{2+\varepsilon}QHN\log^2x+Q^{1/24}H^{47/24}M^{\varepsilon}x^{2-\gamma/24}\log x.
\end{eqnarray}
From (\ref{17})-(\ref{21}), we can get
\begin{eqnarray}
  (\log x)^{-2}|S|^2  \!\! & \ll & \!\! M^{-1/12+\varepsilon}H^2x^2+M^{2+\varepsilon}HNQ+M^{\varepsilon}H^{47/24}x^{2-\gamma/24}Q^{1/24}
                                          \nonumber   \\
  \!\! &    & \!\! + M^{-1/3+\varepsilon}H^{25/12}x^{2+\gamma/12}Q^{-1/12}.
\end{eqnarray}
By Lemma~\ref{Graham-Kolesnik}, we can choose an optimal $Q\in[1,HN\log^{-1}x]$ such that
\begin{eqnarray}
  (\log x)^{-3}|S|^2  \!\! & \ll & \!\! M^{-1/12+\varepsilon}H^2x^2+M^{2+\varepsilon}HN+M^{\varepsilon}H^{47/24}x^{2-\gamma/24}+M^{-1/9+\varepsilon}H^2x^2
                                          \nonumber   \\
  \!\! &    & \!\! + M^{-1/4+\varepsilon}H^2x^{(23+\gamma)/12}+M^{-3/13+\varepsilon}H^2x^{(25+\gamma)/13}.
\end{eqnarray}
Therefore, we have
\begin{eqnarray}\label{SII-upper}
  (\log x)^{-2}S_{II}(M,N) & \ll & M^{-1/24+\varepsilon}H_1x+M^{1/2+\varepsilon}H_1^{1/2}x^{1/2}
                                     \nonumber   \\
  & & +M^{-1/8+\varepsilon}H_1x^{(23+\gamma)/24}+M^{-3/26+\varepsilon}H_1x^{(25+\gamma)/26}
                                       \nonumber   \\
   &   &                        +M^{\varepsilon}H_1^{47/48}x^{1-\gamma/48}  +M^{-1/18+\varepsilon}H_1x
                                       \nonumber   \\
\end{eqnarray}

From (\ref{SII-upper}) we know that, under the condition (\ref{Type-II-M-condition}), the result of Lemma~\ref{type-II} follows.
\end{proof}

\begin{lemma}\label{type-I}
   Suppose that $16(1-\gamma)+16\delta<1,\,\,|a_m|\ll1,\,b_n=1\,\textrm{or}\,\,\,b_n=\log n,\,\,\,MN\asymp x$.
   Let
   \begin{equation}\label{type-1-condition}
     \begin{array}{ll}
        a_1=\displaystyle\frac{3}{2}-19(1-\gamma)-19\delta,
                &   \qquad a_2=\displaystyle\frac{12}{11}-\displaystyle\frac{144}{11}(1-\gamma)-\displaystyle\frac{144}{11}\delta,  \\
        a_3=1-\displaystyle\frac{35}{3}(1-\gamma)-\displaystyle\frac{35}{3}\delta,
                &   \qquad a_4=\displaystyle\frac{18}{17}-\displaystyle\frac{192}{17}(1-\gamma)-\displaystyle\frac{192}{17}\delta,  \\
        a_5=\displaystyle\frac{13}{11}-\displaystyle\frac{118}{11}(1-\gamma)-\displaystyle\frac{118}{11}\delta,
                &   \qquad a_6=\displaystyle\frac{24}{23}-\displaystyle\frac{216}{23}(1-\gamma)-\displaystyle\frac{216}{23}\delta,  \\
        a_7=\displaystyle\frac{26}{29}-\displaystyle\frac{194}{29}(1-\gamma)-\displaystyle\frac{201}{29}\delta,
                &   \qquad a_8=\displaystyle\frac{24}{29}-\displaystyle\frac{180}{29}(1-\gamma)-\displaystyle\frac{186}{29}\delta,  \\
        a_9=\displaystyle\frac{46}{57}-\displaystyle\frac{346}{57}(1-\gamma)-\displaystyle\frac{357}{57}\delta,
                &   \qquad a=\min\big(a_1,a_2,\cdots,a_9\big)-\varepsilon.
     \end{array}
   \end{equation}
   If there holds
   \begin{equation}\label{Type-I-M-condition}
       M\ll x^a,
   \end{equation}
   then we have
   \begin{equation}\label{Type-I-estimation}
     S_{I}(M,N)\ll x^{1-\delta-\varepsilon}.
   \end{equation}
\end{lemma}
\begin{proof}
  Applying partial summation to the inner sum, we have
\begin{eqnarray*}
  &          &  \Bigg|\sum_{m\sim M}\sum_{n\sim N}a_mb_ne\big(\alpha m^3n^3+h(mn+u)^\gamma\big)\Bigg|   \\
 & \leqslant &  \sum_{m\sim M}\Bigg|\sum_{n\sim N}b_ne\big(\alpha m^3n^3+h(mn+u)^\gamma\big)\Bigg|   \\
 & \ll  & (\log x)\sum_{m\sim M}\Bigg|\sum_{n\sim N}e\big(\alpha m^3n^3+h(mn+u)^\gamma\big)\Bigg| =:(\log x)\cdot K_h.
\end{eqnarray*}
Thus, we obtain
\begin{equation*}
   (\log x)^{-1}\cdot S_{I}(M,N) \ll \min\bigg(1,\frac{H_1}{H}\bigg)\sum_{h\sim H}K_h,
\end{equation*}
where
\begin{equation*}
   K_h=\sum_{m\sim M}\Bigg|\sum_{n\sim N}e\big(\alpha m^3n^3+h(mn+u)^\gamma\big)\Bigg|.
\end{equation*}
By H\"{o}lder's inequality, we have
\begin{equation}\label{K_h-8-power}
   K_h^8\ll  M^7 \sum_{m\sim M}\Bigg|\sum_{n\sim N}e\big(\alpha m^3n^3+h(mn+u)^\gamma\big)\Bigg|^8.
\end{equation}

  Suppose $z_n=z_n(m,u,\alpha)=\alpha m^3n^3+h(mn+u)^{\gamma}$. Let $Q,\,J,\,L\,$ be three positive integers, which
  satisfy $1\leqslant Q\leqslant N\log^{-1}x,\,1\leqslant J\leqslant N\log^{-1}x,\,1\leqslant L\leqslant N\log^{-1}x$. Applying
  Lemma~\ref{Heath-Brown-1-lemma-1} to the inner sum of (\ref{K_h-8-power}), we get
\begin{eqnarray*}
   \bigg|\sum_{n\sim N}e(z_n)\bigg|^2
         & \ll & \frac{N}{Q}\sum_{|q|\leqslant Q}\bigg(1-\frac{|q|}{Q}\bigg)\sum_{\substack{n\sim N\\ n+q\sim N}}e(z_{n+q}-z_n)  \\
         & \ll & \frac{N^2}{Q}+\frac{N}{Q}\sum_{1\leqslant q\leqslant Q}\bigg(1-\frac{q}{Q}\bigg)
                 \Bigg|\sum_{\substack{n\sim N\\ n+q\sim N}}e(z_{n+q}-z_n)\Bigg| .
\end{eqnarray*}
Therefore, by Cauchy's inequality, we have
\begin{eqnarray} \label{4-power}
  \bigg|\sum_{n\sim N}e(z_n)\bigg|^4
         & \ll &  \frac{N^4}{Q^2}+\frac{N^2}{Q^2}\Bigg(\sum_{1\leqslant q\leqslant Q}\bigg(1-\frac{q}{Q}\bigg)^2\Bigg)
                  \Bigg(\sum_{1\leqslant q\leqslant Q}\Bigg| \sum_{\substack{n\sim N\\ n+q\sim N}}e(z_{n+q}-z_n) \Bigg|^2\Bigg)
                             \nonumber   \\
         & \ll &   \frac{N^4}{Q^2}+\frac{N^2}{Q}\sum_{1\leqslant q\leqslant Q}\Bigg| \sum_{\substack{n\sim N\\ n+q\sim N}}e(z_{n+q}-z_n) \Bigg|^2.
\end{eqnarray}
Applying Lemma~\ref{Heath-Brown-1-lemma-1} to the inner sum of (\ref{4-power}), we have
\begin{eqnarray}\label{4-power-inner-sum}
     &     &  \Bigg| \sum_{\substack{n\sim N\\ n+q\sim N}}e(z_{n+q}-z_n) \Bigg|^2
                             \nonumber   \\
     & \ll & \frac{N}{J}\sum_{|j|\leqslant J}\bigg(1-\frac{|j|}{J}\bigg)\sum_{\substack{n\sim N,\,n+q\sim N \\n+j\sim N,\,n+q+j\sim N}}
                     \!\!\!\!\!\!e(z_{n+q+j}-z_{n+j}-z_{n+q}+z_n)
                           \nonumber   \\
    & \ll & \frac{N^2}{J}+\frac{N}{J}\sum_{1\leqslant j\leqslant J}\Bigg|\sum_{\substack{n\sim N,\,n+q\sim N \\n+j\sim N,\,n+q+j\sim N}}
               \!\!\!\!e(z_{n+q+j}-z_{n+j}-z_{n+q}+z_n)\Bigg|.
\end{eqnarray}
Putting (\ref{4-power-inner-sum}) into (\ref{4-power}), we have
\begin{eqnarray}
    \bigg|\sum_{n\sim N}e(z_n)\bigg|^4
     & \ll & \frac{N^4}{Q^2}+\frac{N^4}{J}
             +\frac{N^3}{JQ}\sum_{1\leqslant q\leqslant Q}\sum_{1\leqslant j\leqslant J}
                          \nonumber   \\
     &   &          \times \Bigg|\sum_{N<n\leqslant 2N-q-j}e(z_{n+q+j}-z_{n+j}-z_{n+q}+z_n) \Bigg|.
\end{eqnarray}
Therefore, by Cauchy's inequality, we have
\begin{eqnarray}\label{8-power}
   &     &   \bigg|\sum_{n\sim N}e(z_n)\bigg|^8    \nonumber   \\
     & \ll & \!\!\!\!  \frac{N^8}{Q^4}+\frac{N^8}{J^2} +\frac{N^6}{J^2Q^2}
                 \Bigg( \sum_{1\leqslant q\leqslant Q}\sum_{1\leqslant j\leqslant J}
                 \Bigg|\sum_{N<n\leqslant 2N-q-j}e(z_{n+q+j}-z_{n+j}-z_{n+q}+z_n) \Bigg|\Bigg)^2
                       \nonumber   \\
    & \ll & \!\!\!\!    \frac{N^8}{Q^4}+\frac{N^8}{J^2} +\frac{N^6}{JQ}
                  \sum_{1\leqslant q\leqslant Q}\sum_{1\leqslant j\leqslant J}
                       \Bigg|\sum_{N<n\leqslant 2N-q-j}e(z_{n+q+j}-z_{n+j}-z_{n+q}+z_n) \Bigg|^2.
\end{eqnarray}
Set $y_n=y_n(q,j)=z_{n+q+j}-z_{n+j}-z_{n+q}+z_n.$ Applying Lemma~\ref{Heath-Brown-1-lemma-1} to the inner sum of (\ref{8-power}), we have
\begin{eqnarray}\label{8-power-inner-sum}
  &     &   \Bigg|\sum_{N<n\leqslant 2N-q-j}e\big(y_n\big) \Bigg|^2
                      \nonumber   \\
  & \ll & \frac{N}{L}\sum_{|\ell|\leqslant L}\bigg(1-\frac{|\ell|}{L}\bigg)\sum_{\substack{N<n\leqslant 2N-q-j\\N<n+\ell\leqslant 2N-q-j}}
                    e\big(y_{n+\ell}-y_n\big)
                     \nonumber   \\
  & = & \frac{N^2}{L}+ \frac{N}{L}\sum_{1\leqslant|\ell|\leqslant L}\bigg(1-\frac{|\ell|}{L}\bigg)
                 \sum_{\substack{N<n\leqslant 2N-q-j-\ell}} e\big(y_{n+\ell}-y_n\big).
\end{eqnarray}
Putting (\ref{8-power-inner-sum}) into (\ref{8-power}), we have
\begin{eqnarray}\label{8-power-last}
   &     &   \bigg|\sum_{n\sim N}e(z_n)\bigg|^8    \nonumber   \\
     & \ll & \frac{N^8}{Q^4}+\frac{N^8}{J^2} +\frac{N^8}{L}+\frac{N^7}{LJQ}\sum_{q=1}^Q\sum_{j=1}^J\sum_{1\leqslant|\ell|\leqslant L}
                  \bigg(1-\frac{|\ell|}{L}\bigg) \sum_{\substack{N<n\leqslant 2N-q-j-\ell}} e\big(y_{n+\ell}-y_n\big).
                   \nonumber   \\
\end{eqnarray}
Put (\ref{8-power-last}) into (\ref{K_h-8-power}), we obtain
\begin{eqnarray}\label{K_h^8}
   K_h^8 &  \ll    &  \frac{x^8}{Q^4}+\frac{x^8}{J^2}+\frac{x^8}{L}+\frac{x^7}{LJQ}
                  \nonumber   \\
     &     &  \times \sum_{q=1}^Q\sum_{j=1}^J\sum_{\ell=1}^L
                 \Bigg| \sum_{m\sim M} \sum_{N<n\leqslant 2N-q-j-\ell} e\big(y_{n+\ell}-y_n\big) \Bigg|
                   \nonumber   \\
     & =: &  \frac{x^8}{Q^4}+\frac{x^8}{J^2}+\frac{x^8}{L}+\frac{x^7}{LJQ} \sum_{q=1}^Q\sum_{j=1}^J\sum_{\ell=1}^L \big|E_{q,j,\ell} \big|,
\end{eqnarray}
where
\begin{equation}\label{E-qjl}
   E_{q,j,\ell}=\sum_{m\sim M} \sum_{N<n\leqslant 2N-q-j-\ell} e\big(y_{n+\ell}-y_n\big).
\end{equation}
Let
\begin{eqnarray*}
   \Delta\big(n^\gamma;q,j,\ell\big) & = & (n+q+j+\ell)^\gamma-(n+q+j)^{\gamma}-(n+q+\ell)^{\gamma}-(n+j+\ell)^{\gamma}
                                                   \nonumber   \\
                                   &     &  +(n+q)^{\gamma} + (n+j)^{\gamma}+(n+\ell)^{\gamma}-n^{\gamma}.
\end{eqnarray*}
Then we have
\begin{eqnarray*}
            y_{n+\ell}-y_n
 &  =  &   z_{n+q+j+\ell}-z_{n+q+j}-z_{n+q+\ell}-z_{n+j+\ell}+z_{n+q}+z_{n+j}+z_{n+\ell}-z_n
                                \nonumber   \\
  &  =  &   6\alpha qj\ell m^3+\Big(h\big(m(n+q+j+\ell)+u\big)^\gamma-h\big(m(n+q+j)+u\big)^\gamma\Big)
                                \nonumber   \\
  &     &   - \Big(h\big(m(n+q+\ell)+u\big)^\gamma-h\big(m(n+q)+u\big)^\gamma\Big)
                                \nonumber   \\
  &     &   -\Big(h\big(m(n+j+\ell)+u\big)^\gamma-h\big(m(n+j)+u\big)^\gamma\Big)
                                \nonumber   \\
  &     &   +\Big(h\big(m(n+\ell)+u\big)^\gamma-h\big(mn+u\big)^\gamma\Big)
                                \nonumber   \\
  &  =  &   6\alpha qj\ell m^3+ hm^\gamma\Delta\big(n^\gamma;q,j,\ell\big)
                                \nonumber   \\
  &     &   +\gamma h\int_{0}^u\Big(  \big(m(n+q+j+\ell)+t\big)^{\gamma-1}-\big(m(n+q+j)+t\big)^{\gamma-1}\Big)\mathrm{d}t
                                \nonumber   \\
  &     &   -\gamma h\int_{0}^u\Big(  \big(m(n+q+\ell)+t\big)^{\gamma-1}-\big(m(n+q)+t\big)^{\gamma-1}\Big)\mathrm{d}t
                                \nonumber   \\
  &     &   -\gamma h\int_{0}^u\Big(  \big(m(n+j+\ell)+t\big)^{\gamma-1}-\big(m(n+j)+t\big)^{\gamma-1}\Big)\mathrm{d}t
                                \nonumber   \\
  &     &   +\gamma h\int_{0}^u\Big(  \big(m(n+\ell)+t\big)^{\gamma-1}-\big(mn+t\big)^{\gamma-1}\Big)\mathrm{d}t
                                \nonumber   \\
  & =:  &  6\alpha qj\ell m^3+ hm^\gamma\Delta\big(n^\gamma;q,j,\ell\big)+\mathcal{I}_1-\mathcal{I}_2-\mathcal{I}_3+\mathcal{I}_4.
\end{eqnarray*}
By noting that
\begin{equation*}
   \begin{array}{ll}
      \mathcal{I}_1 \asymp hm\ell \displaystyle \int_{0}^u \big(m(n+q+j)+t\big)^{\gamma-2}\mathrm{d}t,
              &   \quad \mathcal{I}_2 \asymp hm\ell \displaystyle\int_{0}^u \big(m(n+q)+t\big)^{\gamma-2}\mathrm{d}t,  \\
      \mathcal{I}_3 \asymp hm\ell \displaystyle \int_{0}^u \big(m(n+j)+t\big)^{\gamma-2}\mathrm{d}t,
              &   \quad \mathcal{I}_4 \asymp hm\ell \displaystyle\int_{0}^u \big(mn+t\big)^{\gamma-2}\mathrm{d}t , \\
   \end{array}
\end{equation*}
we obtain
\begin{eqnarray*}
   &   &  \mathcal{I}_1-\mathcal{I}_2-\mathcal{I}_3+\mathcal{I}_4= (\mathcal{I}_1-\mathcal{I}_2)-(\mathcal{I}_3-\mathcal{I}_4)
                       \nonumber   \\
   & \asymp & hm\ell \int_0^u  \Big(\big(m(n+q+j)+t\big)^{\gamma-2}- \big(m(n+q)+t\big)^{\gamma-2} \Big)\mathrm{d}t
                       \nonumber   \\
   &     &  - hm\ell \int_0^u \Big(\big(m(n+j)+t\big)^{\gamma-2}-\big(mn+t\big)^{\gamma-2}  \Big)\mathrm{d}t
                       \nonumber   \\
   & \asymp &  hm^2j\ell  \int_0^u \Big(\big(m(n+q)+t\big)^{\gamma-3}-\big(mn+t\big)^{\gamma-3}  \Big)\mathrm{d}t
                       \nonumber   \\
   & \asymp &  hm^3qj\ell  \int_0^u \big(mn+t\big)^{\gamma-4} \mathrm{d}t \asymp hqj\ell M^3x^{\gamma-4}.
\end{eqnarray*}
Thus, we get
\begin{eqnarray}\label{y-n-l}
  y_{n+\ell}-y_n
    & = & 6\alpha qj\ell m^3+ hm^\gamma\Delta\big(n^\gamma;q,j,\ell\big)+O\big(hqj\ell M^3x^{\gamma-4}\big)
                      \nonumber   \\
    & =: & G(m,n)++O\big(hqj\ell M^3x^{\gamma-4}\big).
 \end{eqnarray}
Putting (\ref{y-n-l}) into (\ref{E-qjl}), we have
\begin{eqnarray}\label{E-q-j-l-first}
   E_{q,j,\ell} & = & \sum_{m\sim M} \sum_{N<n\leqslant 2N-q-j-\ell}e\Big(G(m,n)+O\big(hqj\ell M^3x^{\gamma-4}\big)\Big)
                   \nonumber   \\
   & = & \sum_{m\sim M} \sum_{N<n\leqslant 2N-q-j-\ell}e\big(G(m,n)\big)\Big(1+O\big(hqj\ell M^3x^{\gamma-4}\big)\Big)
                    \nonumber   \\
   & = &  \sum_{m\sim M} \sum_{N<n\leqslant 2N-q-j-\ell}e\big(G(m,n)\big)+O\big(hqj\ell M^3x^{\gamma-3}\big).
\end{eqnarray}
For any $t\neq0,1$, we have
\begin{eqnarray}\label{Delta-1}
    \Delta\big(n^t;q,j,\ell\big)
     & = &  t\int_0^\ell \big( (n+q+j+\tau)^{t-1}-(n+q+\tau)^{t-1}  \big)\mathrm{d}\tau
                          \nonumber   \\
     &   &  -t\int_0^\ell \big( (n+j+\tau)^{t-1}-(n+\tau)^{t-1}  \big)\mathrm{d}\tau
                          \nonumber   \\
& \asymp &  t(t-1)j\int_0^\ell \big( (n+q+\tau)^{t-2}-(n+\tau)^{t-2}  \big)\mathrm{d}\tau
                          \nonumber   \\
& \asymp &  t(t-1)(t-2)qj\int_0^\ell  (n+\tau)^{t-3}\mathrm{d}\tau
                          \nonumber   \\
 & = & t(t-1)(t-2)qj\ell n^{t-3}
                          \nonumber   \\
 &   &       +t(t-1)(t-2)(t-3)qj \int_0^\ell\int_0^\tau(n+\xi)^{t-4}\mathrm{d}\xi\mathrm{d}\tau
                            \nonumber   \\
 &  =  & t(t-1)(t-2)qj\ell n^{t-3}+O\big( N^{t-4}qj\ell^2\big).
\end{eqnarray}
Similarly, we also have
\begin{equation}\label{Delta-2}
 \Delta\big(n^t;q,j,\ell\big) =t(t-1)(t-2)qj\ell n^{t-3}+O\big( N^{t-4}qj^2\ell\big),
\end{equation}
\begin{equation}\label{Delta-3}
 \Delta\big(n^t;q,j,\ell\big) =t(t-1)(t-2)qj\ell n^{t-3}+O\big( N^{t-4}q^2j\ell\big).
\end{equation}
Combining (\ref{Delta-1}), (\ref{Delta-2}) and (\ref{Delta-3}), we obtain
\begin{eqnarray}\label{Delta-total}
 \Delta\big(n^t;q,j,\ell\big) & = & t(t-1)(t-2)qj\ell n^{t-3}+O\big( N^{t-4}qj\ell(q+j+\ell)\big)
                                          \nonumber   \\
      & = & t(t-1)(t-2)qj\ell n^{t-3}\bigg(1+\bigg(\frac{q+j+\ell}{N}\bigg)\bigg).
\end{eqnarray}
Therefore, it is easy to get
\begin{eqnarray}
\frac{\partial G}{\partial n}  & = & \gamma hm^\gamma\Delta(n^{\gamma-1};q,j,\ell)
            \nonumber    \\
  & = & \gamma(\gamma-1)(\gamma-2)(\gamma-3)hqj\ell m^\gamma n^{\gamma-4}\bigg(1+O\bigg(\frac{q+j+\ell}{N}\bigg)\bigg)
\end{eqnarray}
and
\begin{eqnarray}\label{partial2G}
  \frac{\partial^2 G}{\partial n^2}  & = & \!\!\gamma(\gamma-1) hm^\gamma\Delta(n^{\gamma-2};q,j,\ell)
                        \nonumber    \\
  & = & \!\!\gamma(\gamma-1)(\gamma-2)(\gamma-3)(\gamma-4)hqj\ell m^\gamma n^{\gamma-5}\bigg(1+O\bigg(\frac{q+j+\ell}{N}\bigg)\bigg).
\end{eqnarray}
If $\big|\gamma(\gamma-1)(\gamma-2)(\gamma-3)hqj\ell m^\gamma n^{\gamma-4}\big|\leqslant1/500$, then from (\ref{van-de-corput-1}) of
Lemma~\ref{Jia-0-Lemma-1} we have
\begin{equation}\label{<1/500}
  \sum_{m\sim M}\sum_{N<n\leqslant 2N-q-j-\ell}e\big(G(m,n)\big)
   \ll MN^4(hqj\ell M^\gamma N^\gamma)^{-1}\asymp MN^4(hqj\ell x^\gamma)^{-1}.
\end{equation}
In the rest of this Lemma, we always suppose that $\big|\gamma(\gamma-1)(\gamma-2)(\gamma-3)hqj\ell m^\gamma n^{\gamma-4}\big|>1/500$. By
Lemma \ref{Jia-0-Lemma-3-1}, we have
\begin{eqnarray}
  &   &  \sum_{N<n\leqslant2N-q-j-\ell}e\big(G(m,n)\big)
             \nonumber   \\
  & =  & e\Big(\frac{1}{8}\Big) \sum_{\alpha<\nu\leqslant\beta}\bigg(\frac{\partial^2 G}{\partial n^2}(m,n_{\nu})\bigg)^{-1/2}
          e\big(G(m,n_\nu)-\nu n_\nu\big)+R_1(m,q,j,\ell),
\end{eqnarray}
where
\begin{equation}\label{G=nu}
  \frac{\partial G}{\partial n}(m,n_\nu)=\gamma hm^\gamma \Delta(n_\nu^{\gamma-1};q,j,\ell)=\nu,
\end{equation}
\begin{equation}
  \alpha=\frac{\partial G}{\partial n}(m,N),\quad \beta=\frac{\partial G}{\partial n}(m,2N-q-j-\ell),
\end{equation}
\begin{equation}
R=N^5(hqj\ell x^\gamma)^{-1},\qquad \nu=\frac{\partial G}{\partial n}(m,n_\nu)\asymp hqj\ell m^{\gamma}N^{\gamma-4},
\end{equation}
\begin{equation}
 R_{1}(m,q,j,\ell)\ll \log x+RN^{-1}+\min\bigg(\sqrt{R},\max\bigg(\frac{1}{\|\alpha\|},\frac{1}{\|\beta\|}\bigg)\bigg).
\end{equation}
From Lemma~\ref{Jia-0-Lemma-3-2}, the contribution of $R_{1}(m,q,j,\ell)$ to $E_{q,j,\ell}$ is
\begin{eqnarray}\label{E-qjl-error-contribution}
  & \ll &   \!\!\!M\log x+MRN^{-1}+\sum_{m\sim M}\min\bigg(\sqrt{R},\frac{1}{\|\alpha\|}\bigg)
            +\sum_{m\sim M}\min\bigg(\sqrt{R},\frac{1}{\|\beta\|}\bigg)
               \nonumber \\
  & \ll & \!\!\!M\log x+x^{4-\gamma}(hqj\ell M^3)^{-1}+(hqj\ell)^{1/2}M^{3/2}x^{(\gamma-3)/2}\log x.
\end{eqnarray}
Now, we only need to estimate the exponential sum
\begin{eqnarray}
 &  &    \sum_{m\sim M}\sum_{\alpha<\nu \leqslant\beta}\bigg(\frac{\partial^2 G}{\partial n^2}(m,n_\nu)\bigg)^{-1/2} e\big(G(m,n_\nu)-\nu n_\nu\big)
                \nonumber \\
 &  = &   \sum_{\nu}\sum_{m\in\mathfrak{I}_\nu}\bigg(\frac{\partial^2 G}{\partial n^2}(m,n_\nu)\bigg)^{-1/2} e\big(G(m,n_\nu)-\nu n_\nu\big),
\end{eqnarray}
where $\mathfrak{I}_\nu$ is a subinterval of $(M,2M]$.

  For fixed $\nu$, define $\Delta_\lambda=\Delta(n_\nu^{\lambda};q,j,\ell)$, where $\lambda$ is arbitrary real number. Taking derivative of $m$ on both sides of the equation (\ref{G=nu}), we have
\begin{equation}
  n_\nu'=-\frac{\gamma\Delta_{\gamma-1}}{(\gamma-1)m\Delta_{\gamma-2}}.
\end{equation}
Combining (\ref{Delta-1}) and (\ref{partial2G}), we get
\begin{eqnarray}
   &    &  \frac{\mathrm{d}}{\mathrm{d}m}\bigg(\frac{\partial^2 G}{\partial n^2}(m,n_\nu)\bigg)
              \nonumber   \\
     & = & \frac{\gamma^2hm^{\gamma-1}}{\Delta_{\gamma-2}} \Big((\gamma-1)\Delta_{\gamma-2}^2-(\gamma-2)\Delta_{\gamma-1}\Delta_{\gamma-3}\Big)
            \nonumber   \\
     & = & \gamma^2(\gamma-1)(\gamma-2)(\gamma-3)hqj\ell m^{\gamma-1}n_\nu^{\gamma-5}\bigg(1+O\bigg(\frac{q+j+\ell}{N}\bigg)\bigg),
\end{eqnarray}
so that $\big(\frac{\partial^2 G}{\partial n^2}(m,n_\nu)\big)^{-1/2}$ is monotonic in $m$.

  Let $g(m)=G\big(m,n_\nu(m)\big)-\nu n_{\nu}(m)$. By a series of simple calculation, we obtain
\begin{eqnarray}
  g'(m) & = & 18\alpha qj\ell m^2+\gamma hm^{\gamma-1}\Delta_\gamma,
                    \nonumber   \\
  g''(m) & = & 36\alpha qj\ell m+\frac{\gamma h}{\gamma-1}\cdot
               \frac{(\gamma-1)^2\Delta_\gamma\Delta_{\gamma-2}-\gamma^2\Delta_{\gamma-1}^2}{m^{2-\gamma}\Delta_{\gamma-2}}
                         \nonumber   \\
        & =: &  36\alpha qj\ell m+\frac{\gamma h}{\gamma-1}\cdot \frac{g_1(m)-g_2(m)}{g_0(m)},
\end{eqnarray}
where
\begin{equation*}
  g_1(m)=(\gamma-1)^2\Delta_\gamma\Delta_{\gamma-2},\quad g_2(m)=\gamma^2\Delta_{\gamma-1}^2,\quad g_0(m)=m^{2-\gamma}\Delta_{\gamma-2}.
\end{equation*}
Hence
\begin{equation}\label{g-3-dirivative}
  \!\!\!\!\!\! \!\!\!\!\!\!  g'''(m)=36\alpha qj\ell +\frac{\gamma h}{\gamma-1}\cdot\frac{(g_1'(m)-g_2'(m))g_0(m)-g_0'(m)(g_1(m)-g_2(m))}{g_0^2(m)},
\end{equation}
where
\begin{equation}\label{three-derivative}
  \begin{array}{l}
     g_1'(m)=(\gamma-1)^2\big(\gamma\Delta_{\gamma-1}\Delta_{\gamma-2}+(\gamma-2)\Delta_{\gamma}\Delta_{\gamma-3}\big)n_\nu'(m) , \\
     g_2'(m)=2\gamma^2(\gamma-1)\Delta_{\gamma-1}\Delta_{\gamma-2}n_\nu'(m) ,  \\
     g_0'(m)=\displaystyle\frac{(2-\gamma)m^{1-\gamma}}{(\gamma-1)\Delta_{\gamma-2}} \big((\gamma-1)\Delta_{\gamma-2}^2+\gamma\Delta_{\gamma-1}\Delta_{\gamma-3}\big).
  \end{array}
\end{equation}
Putting (\ref{three-derivative}) into (\ref{g-3-dirivative}), we get
\begin{eqnarray}
  g'''(m) & = & 36\alpha qj \ell+\frac{\gamma h}{(\gamma-1)^2}
                  \nonumber   \\
   &  & \times \frac{3\gamma^2(\gamma-1)\Delta_{\gamma-1}^2\Delta_{\gamma-2}^2+(\gamma-1)^3(\gamma-2)\Delta_{\gamma}\Delta_{\gamma-2}^3-\gamma^3(\gamma-2)   \Delta_{\gamma-1}^3\Delta_{\gamma-3}   }{m^{3-\gamma}\Delta_{\gamma-2}^3}
                  \nonumber   \\
   &  =: &  36\alpha qj \ell+\frac{\gamma h}{(\gamma-1)^2}\cdot \frac{g_3(m)-g_4(m)}{g_5(m)},
\end{eqnarray}
where
\begin{equation*}
   \begin{array}{l}
     g_3(m)=3\gamma^2(\gamma-1)\Delta_{\gamma-1}^2\Delta_{\gamma-2}^2+(\gamma-1)^3(\gamma-2)\Delta_{\gamma}\Delta_{\gamma-2}^3, \\
     g_4(m)=\gamma^3(\gamma-2)   \Delta_{\gamma-1}^3\Delta_{\gamma-3} ,\quad g_5(m)=m^{3-\gamma}\Delta_{\gamma-2}^3.
   \end{array}
\end{equation*}
Hence
\begin{equation}\label{g-4-dirivative}
  \!\!\!\!\!\! \!\!\!\!\!\!  g^{(4)}(m)=\frac{\gamma h}{(\gamma-1)^2}\cdot\frac{(g_3'(m)-g_4'(m))g_5(m)-g_5'(m)(g_3(m)-g_4(m))}{g_5^2(m)},
\end{equation}
where
\begin{eqnarray}\label{four-derivative}
   &  &   g_3'(m)=\Big(\gamma(\gamma-1)^2(\gamma+1)(\gamma+2)\Delta_{\gamma-1}\Delta_{\gamma-2}^3
              +6\gamma^2(\gamma-1)(\gamma-2)\Delta_{\gamma-1}^2\Delta_{\gamma-2}\Delta_{\gamma-3}
                   \nonumber  \\
    &  &       \qquad\qquad \qquad\qquad  +3(\gamma-1)^3(\gamma-2)^2 \Delta_{\gamma} \Delta_{\gamma-2}^2  \Delta_{\gamma-3}    \Big)n_\nu'(m) ,
                   \nonumber  \\
 &  &    g_4'(m)=\Big(3\gamma^3(\gamma-1)(\gamma-2)\Delta_{\gamma-1}^2
                \Delta_{\gamma-2}\Delta_{\gamma-3}+\gamma^3(\gamma-2)(\gamma-3)\Delta_{\gamma-1}^3\Delta_{\gamma-4}\Big)  n_\nu'(m) ,
                   \nonumber  \\
  &  &    g_5'(m)=\displaystyle\frac{m^{2-\gamma}}{\gamma-1}
      \Big((\gamma-1)(3-\gamma)\Delta_{\gamma-2}^3-3\gamma(\gamma-2)\Delta_{\gamma-1}\Delta_{\gamma-2}\Delta_{\gamma-3}\Big).
\end{eqnarray}
Put (\ref{four-derivative}) into (\ref{g-4-dirivative}), we obtain
\begin{eqnarray}
  g^{(4)}(m) & = & -\frac{\gamma h}{(\gamma-1)^3}\cdot \frac{1}{m^{4-\gamma}\Delta^5_{\gamma-2}}
                   \Big( \gamma^2(\gamma-1)^2 (\gamma^2+11)\Delta_{\gamma-1}^2\Delta_{\gamma-2}^4
                           \nonumber   \\
    &    &   -2\gamma^3(\gamma-1)(\gamma-2)(\gamma+3) \Delta_{\gamma-1}^3\Delta_{\gamma-2}^2 \Delta_{\gamma-3}
                           \nonumber   \\
    &    &   -r^{4}(\gamma-2) (\gamma-3) \Delta_{\gamma-1}^4\Delta_{\gamma-2}\Delta_{\gamma-4}
                           \nonumber   \\
    &    &   -(\gamma-1)^4(\gamma-2)(\gamma-3) \Delta_{\gamma} \Delta_{\gamma-2}^5 +3\gamma^4(\gamma-2)^2\Delta_{\gamma-1}^4\Delta_{\gamma-3}^2 \Big).
\end{eqnarray}
Combining (\ref{Delta-total}), we have
\begin{equation}
  g^{(4)}(m)=c_0(\gamma) hqj\ell m^{\gamma-4}n_\nu^{\gamma-3}\bigg(1+\bigg(\frac{q+j+\ell}{N}\bigg)\bigg)\asymp hqj\ell M^{-1}x^{\gamma-3},
\end{equation}
where $c_0(\gamma)=-8\gamma^2(\gamma-1)(\gamma-2)^2(\gamma-3)(\gamma-4)^{-3}(3\gamma-8)$.

By partial summation and Lemma \ref{Heath-Brown-3-Theorem-1} with parameter $k=4$, we obtain
\begin{eqnarray}\label{E-qjl-main}
  &   &  \sum_{\nu}\sum_{m\in \mathfrak{I}_\nu}\Bigg(\frac{\partial^2G}{\partial n^2}(m,n_\nu)\bigg)^{-1/2}
         e\big(G(m,n_\nu)-\nu n_\nu\big)
              \nonumber   \\
  & \ll & M^{1+\varepsilon}\Big( \big(hqj\ell M^{-1}x^{\gamma-3}\big)^{1/12}+ M^{-1/12}+M^{-1/6}\big(hqj\ell M^{-1}x^{\gamma-3}\big)^{-1/24}\Big)
              \nonumber   \\
  &    & \times \big(hqj\ell M^\gamma N^{\gamma-5}\big)^{-1/2}\cdot hqj\ell M^{\gamma}N^{\gamma-4}
             \nonumber   \\
  & \ll & \big(hqj\ell \big)^{7/12}M^{29/12+\varepsilon}x^{7(\gamma-3)/12} +\big(hqj\ell \big)^{1/2}M^{29/12+\varepsilon}x^{(\gamma-3)/2}
              \nonumber   \\
  &    &   +\big(hqj\ell \big)^{11/24}M^{19/8+\varepsilon}x^{11(\gamma-3)/24}.
\end{eqnarray}
From (\ref{E-q-j-l-first}), (\ref{<1/500}), (\ref{E-qjl-error-contribution}) and (\ref{E-qjl-main}), we get
\begin{eqnarray}\label{E-qjl-upper-bound}
     &     &   (\log x)^{-1}\cdot E_{q,j,\ell}
                                                 \nonumber   \\
     & \ll &  hqj\ell M^3x^{\gamma-3}+M+(hqj\ell M^3)^{-1}x^{4-\gamma}+(hqj\ell)^{1/2}M^{29/12+\varepsilon}x^{(\gamma-3)/2}
                                                \nonumber   \\
        &    &   +(hqj\ell)^{7/12}M^{29/12+\varepsilon}x^{7(\gamma-3)/12}+(hqj\ell)^{11/24}M^{19/8+\varepsilon}x^{11(\gamma-3)/24}.
\end{eqnarray}
Putting (\ref{E-qjl-upper-bound}) into (\ref{K_h^8}), we get
\begin{eqnarray}\label{K_h^8-log-4}
 (\log x)^{-4}\cdot K_{h}^8  & \ll &  x^8Q^{-4}+x^8J^{-2}+x^8L^{-1}+Mx^7+\big(hQJLM^3\big)^{-1}x^{11-\gamma}
                                            \nonumber   \\
  &   & +hQJLM^3x^{\gamma+4}+\big(hQJL\big)^{7/12}M^{29/12+\varepsilon}x^{7(\gamma+9)/12}
                                             \nonumber   \\
  &   & +\big(hQJL\big)^{1/2}M^{29/12+\varepsilon}x^{(\gamma+11)/2}
                                             \nonumber   \\
  &   &  +\big(hQJL\big)^{11/24}M^{19/8+\varepsilon}x^{11\gamma/24+45/8}.
\end{eqnarray}
Next, we apply Lemma \ref{Graham-Kolesnik} to (\ref{K_h^8-log-4}) in $Q,\,J,\,L$ one step at a time. First, for fixed $Q$ and $J$, we choose an optimal $L\in[1,N\log^{-1}x]$ and obtain
\begin{eqnarray}
   (\log x)^{-4}\cdot K_{h}^8  & \ll & x^{15/2}+Mx^7\log x+M^{8/19+\varepsilon}x^{140/19}+M^{11/18+\varepsilon}x^{22/3}
                                           \nonumber   \\
   &  & +M^{24/35+\varepsilon}x^{256/35}+x^8Q^{-4}+\big(hQJ\big)^{-1}M^{-2}x^{10-\gamma}\log x+x^8J^{-2}
                                           \nonumber   \\
   &  & +hQJM^3x^{\gamma+4}   +\big(hQJ\big)^{7/12}M^{29/12+\varepsilon}x^{7(\gamma+9)/12}
                                          \nonumber   \\
   &  & +\big(hQJ\big)^{1/2}M^{29/12+\varepsilon}x^{(\gamma+11)/2} +\big(hQJ\big)^{11/24}M^{19/8+\varepsilon}x^{11\gamma/24+45/8}
                                            \nonumber   \\
   &  & +\big(hQJ\big)^{1/3}M^{29/18+\varepsilon}x^{(\gamma+19)/3}  + \big(hQJ\big)^{11/35}M^{57/35+\varepsilon}x^{(11\gamma+223)/35}
                                               \nonumber   \\
   &  & +\big(hQJ\big)^{1/2}M^{3/2}x^{\gamma/2+6}  +\big(hQJ\big)^{7/19}M^{29/19+\varepsilon}x^{7(\gamma+17)/19}.
\end{eqnarray}
Second, for fixed $Q$, we choose an optimal $J\in[1,N\log^{-1}x]$ and obtain
\begin{eqnarray*}
           (\log x)^{-6}\cdot K_{h}^8
    & \ll & x^{15/2}+M^{1+\varepsilon}x^7+M^{8/19+\varepsilon}x^{140/19}+M^{11/18+\varepsilon}x^{22/3}
                                               \nonumber   \\
    &   &    +M^{24/35+\varepsilon}x^{256/35}+M^{15/26+\varepsilon}x^{189/26}+M^{17/24+\varepsilon}x^{29/4}
                                                \nonumber   \\
    &    &   +M^{35/46+\varepsilon}x^{333/46}+h^{-1}M^{-1}x^{9-\gamma}Q^{-1}+x^8Q^{-4}
\end{eqnarray*}                                                  
\begin{eqnarray}
    &    &   +hM^3x^{\gamma+4}Q+h^{7/12}M^{29/12+\varepsilon}x^{7(\gamma+9)/12}Q^{7/12}
                                                \nonumber   \\
    &    &   +h^{1/2}M^{29/12+\varepsilon}x^{(\gamma+11)/2}Q^{1/2} + h^{11/24}M^{19/8+\varepsilon}x^{11\gamma/24+45/8}Q^{11/24}
                                                \nonumber   \\
    &   &  +h^{1/2}M^{3/2}x^{\gamma/2+6}Q^{1/2}+h^{7/19}M^{29/19+\varepsilon}x^{7(\gamma+17)/19}Q^{7/19}
                                                   \nonumber   \\
    &   &  +h^{1/3}M^{29/18+\varepsilon}x^{(\gamma+19)/3}Q^{1/3}  + h^{11/35}M^{57/35+\varepsilon}x^{(11\gamma+223)/35}Q^{11/35}
                                                  \nonumber   \\
    &   &  +h^{2/3}M^2x^{2(\gamma+8)/3}Q^{2/3} + h^{14/31}M^{58/31+\varepsilon}x^{14(\gamma+13)/31}Q^{14/31}
                                                     \nonumber   \\
    &   &  +h^{2/5}M^{29/15+\varepsilon}x^{2\gamma/5+6}Q^{2/5} +  h^{22/59}M^{114/59+\varepsilon}x^{(22\gamma+358)/59}Q^{22/59}
                                                  \nonumber   \\
    &   & + h^{2/5}M^{6/5}x^{2(\gamma+16)/5}Q^{2/5}  + h^{14/45}M^{58/45+\varepsilon}x^{14(\gamma+21)/45}Q^{14/45}
                                                  \nonumber   \\
    &  & +h^{2/7}M^{29/21+\varepsilon}x^{2(\gamma+23)/7}Q^{2/7}  +  h^{22/81}M^{38/27+\varepsilon}x^{(22\gamma+534)/81}Q^{22/81}.
\end{eqnarray}
Finally, we choose an optimal $Q\in[1,N\log^{-1}x]$ and obtain
\begin{eqnarray}\label{SI-upper}
   &     &     (\log x)^{-2}\cdot K_h  
                     \nonumber   \\
   & \ll &  x^{15/16}+M^{1/8+\varepsilon}x^{7/8}+M^{1/19+\varepsilon}x^{35/38}+M^{11/144+\varepsilon}x^{11/12}
                                                 \nonumber   \\
   &   &  +M^{3/35+\varepsilon}x^{32/35}+M^{17/192+\varepsilon}x^{29/32}+M^{35/368+\varepsilon}x^{333/368}
                                                \nonumber   \\
   &   &   +M^{1/2}x^{1/2}+M^{23/140+\varepsilon}x^{117/140}+M^{23/162+\varepsilon}x^{139/162}
                                                \nonumber   \\
   &   &   +M^{1/14}x^{25/28}+M^{11/118+\varepsilon}x^{105/118}+M^{23/216+\varepsilon}x^{8/9}
                                                \nonumber   \\
   &   &  +M^{23/206+\varepsilon}x^{183/206}+x^{1-\gamma/8}+h^{1/8}M^{3/8}x^{(\gamma+4)/8}
                                                 \nonumber   \\
   &   &  +h^{7/96}M^{29/96+\varepsilon}x^{7(\gamma+9)/96}+h^{1/16}M^{29/96+\varepsilon}x^{(\gamma+11)/16}
                                                  \nonumber   \\
   &   &  +h^{11/192}M^{19/64+\varepsilon}x^{11\gamma/192+45/64}  +  h^{1/16}M^{3/16}x^{(\gamma+12)/16}
                                                   \nonumber   \\
   &   &   +h^{7/152}M^{29/152+\varepsilon}x^{7(\gamma+17)/152} +h^{1/24}M^{29/144+\varepsilon}x^{(\gamma+19)/24}
                                                    \nonumber   \\
   &   &   +h^{11/280}M^{57/280+\varepsilon}x^{(11\gamma+223)/280}  +h^{1/12}M^{1/4}x^{(\gamma+8)/12}
                                                \nonumber   \\
   &   &   +h^{7/124}M^{29/124+\varepsilon}x^{7(\gamma+13)/124}+ h^{1/20}M^{29/120+\varepsilon}x^{(\gamma+15)/20}
                                                 \nonumber   \\
   &   &  +h^{11/236}M^{57/236+\varepsilon}x^{(11\gamma+179)/236}  +  h^{1/20}M^{3/20}x^{(\gamma+16)/20}
                                                  \nonumber   \\
   &   &  +h^{7/180}M^{29/180+\varepsilon}x^{7(\gamma+21)/180} +  h^{1/28}M^{29/168+\varepsilon}x^{(\gamma+23)/28}
                                                    \nonumber   \\
   &   &  +h^{11/324}M^{19/108+\varepsilon}x^{(11\gamma+267)/324}  + h^{1/10}M^{3/10}x^{(\gamma+6)/10}
                                                   \nonumber   \\
   &   &  +h^{7/110}M^{29/110+\varepsilon}x^{7(\gamma+11)/110}  +h^{1/18}M^{29/108+\varepsilon}x^{(\gamma+13)/18}
                                                    \nonumber   \\
   &   &  +h^{11/214}M^{57/214+\varepsilon}x^{(11\gamma+157)/214} +  h^{1/18}M^{1/6}x^{(\gamma+14)/18}
                                                      \nonumber   \\
   &   &  +h^{7/166}M^{29/166+\varepsilon}x^{7(\gamma+19)/166} + h^{1/26}M^{29/156+\varepsilon}x^{(\gamma+21)/26}
                                                    \nonumber   \\
   &   &  +h^{11/302}M^{57/302+\varepsilon}x^{(11\gamma+245)/302} +  h^{1/14}M^{3/14}x^{(\gamma+10)/14}
                                                     \nonumber   \\
   &   &  +h^{7/138}M^{29/138+\varepsilon}x^{7(\gamma+15)/138} + h^{1/22}M^{29/132+\varepsilon}x^{(\gamma+17)/22}
                                                     \nonumber   \\
   &   &  +h^{11/258}M^{57/258+\varepsilon}x^{(11\gamma+201)/258} + h^{1/22}M^{3/22}x^{(\gamma+18)/22}
                                                   \nonumber   \\
   &   &  +h^{7/194}M^{29/194+\varepsilon}x^{7(\gamma+23)/194}  + h^{1/30}M^{29/180+\varepsilon}x^{(\gamma+25)/30}
                                                    \nonumber   \\
   &   &  +h^{11/346}M^{57/346+\varepsilon}x^{(11\gamma+289)/346}.
\end{eqnarray}

 From (\ref{SI-upper}) we know that, under the condition (\ref{type-1-condition}), the result of Lemma~\ref{type-I} follows.
\end{proof}

\section{Proof of Theorem~\ref{Hua's-theorem-k=3}}

 In order to prove Theorem \ref{Hua's-theorem-k=3}, it is sufficient for us to prove the following proposition.
\begin{proposition}\label{proposition-T3-S3}
   Suppose that $0<\gamma\leqslant1,\,\delta>0$ and
   \begin{equation*}
       1714(1-\gamma)+1725\delta<46.
   \end{equation*}
Then, uniformly in $\alpha$, we have
\begin{equation*}
   T_{3}(N,\alpha)=S_{3}(N,\alpha)+O\big(P^{1-\delta-\varepsilon}\big),
\end{equation*}
 where the implied constant depends only on $\gamma$ and $\delta$.
\end{proposition}

\subsection{Proof of Proposition~\ref{proposition-T3-S3}}

 We have
 \begin{eqnarray*}
      \frac{1}{\gamma}\sum_{\substack{p\leqslant P\\ p\in\mathcal{P}_\gamma}} p^{1-\gamma}e\big(\alpha p^3\big)
     & = & \frac{1}{\gamma}\sum_{p\leqslant P} p^{1-\gamma}e\big(\alpha p^3\big)\big([-p^\gamma]-[-(p+1)^\gamma]\big) \\
 \end{eqnarray*}
 \begin{eqnarray*}
     & = & \sum_{p\leqslant P}e\big(\alpha p^3\big)+\frac{1}{\gamma}\sum_{p\leqslant P}p^{1-\gamma}e(\alpha p^3)
            \big(\psi(-(p+1)^{\gamma})-\psi(-p^\gamma)\big)+O(\log P).
 \end{eqnarray*}
By noting that, for $p\sim x$ satisfying $x\leqslant P^{1/2}$, we have
\begin{equation*}
  \sum_{p\sim x}p^{1-\gamma}e\big(\alpha p^3\big)\big(\psi(-(p+1)^{\gamma})-\psi(-p^\gamma)\big)
  \ll\sum_{p\sim x}p^{1-\gamma} \ll x^{2-\gamma}\ll P^{1-\gamma/2}\ll P^{1-\delta-\varepsilon}.
\end{equation*}
Therefore, in order to prove Proposition~\ref{proposition-T3-S3}, it is sufficient for us to prove that, for
any $x$ satisfying $P^{1/2}<x\leqslant P$, there holds
\begin{equation*}
 \sum_{p\sim x}p^{1-\gamma}e\big(\alpha p^3\big)\big(\psi(-(p+1)^{\gamma})-\psi(-p^\gamma)\big)\ll x^{1-\delta-\varepsilon}.
\end{equation*}
By partial summation, we have
\begin{eqnarray*}
  &   &  \sum_{p\sim x}p^{1-\gamma}e\big(\alpha p^3\big)\big(\psi(-(p+1)^{\gamma})-\psi(-p^\gamma)\big)  \\
  & \ll & (\log x)^{-1}\Bigg|\sum_{n\sim x}\Lambda(n)n^{1-\gamma}e\big(\alpha n^3\big)\big(\psi(-(n+1)^{\gamma})-\psi(-n^\gamma)\big)\Bigg|
          +x^{3/2-\gamma}\log x.
\end{eqnarray*}
Hence, we only need to show that
\begin{eqnarray}\label{only-1}
  \sum_{n\sim x}\Lambda(n)n^{1-\gamma}e\big(\alpha n^3\big)\big(\psi(-(n+1)^{\gamma})-\psi(-n^\gamma) \ll x^{1-\delta-\varepsilon}.
\end{eqnarray}

 Applying Lemma~\ref{Heath-Brown-1-lemma-2-1} to (\ref{only-1}) with the parameter $H=H_0$, the contribution of the error term in (\ref{psi(theta)}) is
\begin{eqnarray*}
  & \ll & \sum_{n\sim x}\Lambda(n)n^{1-\gamma}\min\bigg(1,\frac{1}{H_0\|n^\gamma\|}\bigg) \\
  & \ll & x^{1-\gamma}\log x\sum_{n\sim x}\min\bigg(1,\frac{1}{H_0\|n^\gamma\|}\bigg) \\
  & \ll & x^{1-\gamma}\log x\sum_{n\sim x}\sum_{h=-\infty}^{\infty}a_he\big(hn^\gamma\big) \\
  & \ll & x^{1-\gamma}\log x\sum_{h=-\infty}^{\infty}|a_h| \Bigg|\sum_{n\sim x}e\big(hn^\gamma\big)\Bigg|.
\end{eqnarray*}
For $h\neq0$, applying (\ref{van-de-corput-2}) to the inner sum, we get
\begin{equation*}
  \bigg|\sum_{n\sim x}e\big(hn^\gamma\big)\bigg|\ll h^{1/2}x^{\gamma/2}+h^{-1}x^{1-\gamma}.
\end{equation*}
Therefore, the contribution of the error term is
\begin{eqnarray}\label{error-contribution}
   & \ll & x^{1-\gamma}\log x\bigg(\frac{x\log H_0}{H_0}+\sum_{\substack{h=-\infty\\h\neq0}}^\infty
                   |a_h|\Big(|h|^{1/2}x^{\gamma/2}+|h|^{-1}x^{1-\gamma}\Big)\bigg)
                   \nonumber   \\
   & \ll & x^{1-\gamma}\log x\bigg(\frac{x\log H_0}{H_0}+\sum_{0<h\leqslant H_0}\frac{1}{h}\Big(h^{1/2}x^{\gamma/2}+h^{-1}x^{1-\gamma}\Big)
                   \nonumber   \\
   &     &   \qquad\qquad \qquad\qquad\qquad +\sum_{h>H_0}\frac{H_0}{h^2}\Big(h^{1/2}x^{\gamma/2}+h^{-1}x^{1-\gamma}\Big)\bigg)
                   \nonumber   \\
   & \ll & x^{1-\gamma}\log^2 x\Big(xH_0^{-1}+x^{1-\gamma}+H_0^{1/2}x^{\gamma/2}\Big).
\end{eqnarray}
Taking $H_0=x^{1-\gamma+\delta+\varepsilon}$, then (\ref{error-contribution}) is $\ll x^{1-\delta-\varepsilon}$.

  From (\ref{only-1}) and (\ref{error-contribution}), it is easy to see that, in order to prove Proposition \ref{proposition-T3-S3}, we only need to prove
\begin{equation} \label{only-2}
  \sum_{h\sim H}\frac{1}{h}\Bigg|\sum_{n\sim x}\Lambda(n)n^{1-\gamma}e\big(\alpha n^3\big)
        \Big( e\big(h(n+1)^\gamma\big)-e\big(hn^\gamma\big) \Big)  \Bigg| \ll x^{1-\delta-\varepsilon }.
\end{equation}
Set $H_1=x^{1-\gamma}$. If $H\leqslant H_1$, we write
\begin{equation}\label{H<H1}
   e\big(h(n+1)^\gamma\big)-e\big(hn^\gamma\big) = 2\pi ih\gamma\int_{0}^1 (n+u)^{\gamma-1} e\big(h(n+u)^\gamma\big)\mathrm{d}u.
\end{equation}
Putting (\ref{H<H1}) into the left hand side of (\ref{only-2}) and combining partial summation, we can see that the left hand of (\ref{only-2}) is
\begin{eqnarray}\label{H<H1-estimate}
    \ll  \sum_{h\sim H}\Bigg|\sum_{n\sim x}\Lambda(n)e\Big(\alpha n^3+h(n+u)^\gamma\Big)\Bigg|.
\end{eqnarray}
If $H>H_1$, we divide the left hand side of (\ref{only-2}) into two parts and treat them separately.
Applying partial summation to the inner sum
of the left hand side of (\ref{only-2}), we can see that the left hand of (\ref{only-2}) is
\begin{eqnarray}\label{H>H1-estimate}
    \ll  \frac{H_1}{H}\sum_{h\sim H}\Bigg|\sum_{n\sim x}\Lambda(n)e\Big(\alpha n^3+h(n+u)^\gamma\Big)\Bigg|.
\end{eqnarray}
Combining (\ref{only-2}), (\ref{H<H1-estimate}) and (\ref{H>H1-estimate}), it is sufficient to show that
\begin{equation*}
   \min\bigg(1,\frac{H_1}{H}\bigg)\sum_{h\sim H}\Bigg|\sum_{n\sim x}\Lambda(n)e\Big(\alpha n^3+h(n+u)^\gamma\Big)\Bigg|\ll x^{1-\delta-\varepsilon}.
\end{equation*}

  Take parameters $a_1,\cdots,a_9$ as condition (\ref{type-1-condition}) in Lemma \ref{type-I}. Let
\begin{equation*}
   a=\min\big(a_1,\cdots,a_9\big)-\varepsilon,\quad b=24(1-\gamma)+24\delta+\varepsilon,\quad c=\gamma-2\delta-\varepsilon.
\end{equation*}
Obviously, it is easy to check that
\begin{equation*}
    b<2/3,\qquad b<a,\qquad 1-c<c-b.
\end{equation*}
By Lemma~\ref{Heath-Brown-identity} with $k=3$, one can see that the exponential sum
\begin{equation*}
   \min\bigg(1,\frac{H_1}{H}\bigg)\sum_{h\sim H}\Bigg|\sum_{n\sim x}\Lambda(n)e\big(\alpha n^3+h(n+u)^\gamma\big)\Bigg|
\end{equation*}
can be written as linear combination of $O\big(\log^6x\big)$ sums of the form
\begin{eqnarray}\label{mathcal-T}
  \mathcal{T}  & = & \min\bigg(1,\frac{H_1}{H}\bigg)\sum_{h\sim H}\Bigg|\sum_{n_1\sim N_1}\cdots\sum_{n_6\sim N_6}(\log n_1)\mu(n_4)\mu(n_5)\mu(n_6)
                          \nonumber     \\
  &  &  \qquad \qquad\qquad\qquad\qquad\qquad\times
          e\Big(\alpha (n_1\cdots n_6)^3+h(n_1\cdots n_6+u)^\gamma\Big)\Bigg|,
\end{eqnarray}
where $N_1\cdots N_6\asymp x;\,2N_i\leqslant (2x)^{1/3},\,i=4,5,6$ and some $n_i$ may only take value $1$. Therefore, it is sufficient for us to prove that, for each $\mathcal{T}$ defined as (\ref{mathcal-T}), there holds $\mathcal{T}\ll x^{1-\delta-\varepsilon}$. Next, we will consider three
cases.

\textbf{Case 1}\quad If there exists an $N_j$ such that $N_j\geqslant x^{1-b}$, then we must have $j\leqslant3$ for the
fact that $1-b>1/3$. Let $m=\prod\limits_{i\neq j}n_i,\,n=n_j,\, M=\prod\limits_{i\neq j}N_i,\, N=N_j$. In this case, we can see
that $\mathcal{T}$ is a sum of ``Type I" satisfying $M\ll x^b\ll x^a$. By Lemma \ref{type-I}, the result follows.

\textbf{Case 2}\quad If there exists an $N_j$ such that $x^{1-c}\leqslant N_j <x^{1-b}$, then we
take $m=\prod\limits_{i\neq j}n_i,\,n=n_j,\, M=\prod\limits_{i\neq j}N_i,\, N=N_j$. Thus, $\mathcal{T}$ is a sum
of ``Type II" satisfying $x^b\ll M\ll x^c$. By Lemma \ref{type-II}, the result follows.

\textbf{Case 3}\quad If $N_j<x^{1-c}\,(j=1,2,3,4,5,6)$, without loss of generality, we assume
that $N_1\geqslant N_2\geqslant\cdots \geqslant N_6$. Let $\ell$ denote the  natural number $j$ such that
\begin{equation*}
  N_1N_2\cdots N_{j-1}<x^{1-c} ,\qquad N_1N_2\cdots N_{j}\geqslant x^{1-c}.
\end{equation*}
Since $N_1<x^{1-c}$ and $N_6<x^{1-c}$, then $2\leqslant \ell\leqslant5$. Thus, we have
\begin{equation*}
  x^{1-c}\leqslant N_1N_2\cdots N_\ell=(N_1\cdots N_{\ell-1})\cdot N_\ell<x^{1-c}\cdot x^{1-c}<x^{1-b}.
\end{equation*}
Let $m=\prod\limits_{i=\ell+1}^6n_i,\,n=\prod\limits_{i=1}^{\ell}n_i,\,M=\prod\limits_{i=\ell+1}^6N_i,\,N=\prod\limits_{i=1}^{\ell}N_i$. At this time,
$\mathcal{T}$ is a sum of ``Type II" satisfying $x^b\ll M\ll x^c$. By Lemma \ref{type-II}, the result follows.

   Combining the above three cases, we can assert that Proposition \ref{proposition-T3-S3} holds.

   \subsection{Proof of Theorem~\ref{Hua's-theorem-k=3}}

 Take parameters as follows:
\begin{equation*}
    \mathcal{Q}=N^{\sigma},\qquad \tau=N^{1-\sigma},
\end{equation*}
where $\sigma$ satisfies $0<\sigma\leqslant1/6$ to be determined later. When $1\leqslant a\leqslant q\leqslant \mathcal{Q}$ and $(a,q)=1$, define
major arcs and minor arcs as following:
\begin{equation*}
   \mathfrak{M}(a,q)=\bigg[\frac{a}{q}-\frac{1}{q\tau},\frac{a}{q}+\frac{1}{q\tau}\bigg]
\end{equation*}
and
\begin{equation*}
   \mathfrak{M}=\bigcup_{q\leqslant \mathcal{Q}}\bigcup_{\substack{1\leqslant a\leqslant q\\(a,q)=1}}\mathfrak{M}(a,q),
   \qquad  \mathfrak{m}=\bigg[\frac{1}{\tau},1+\frac{1}{\tau}\bigg]\setminus\mathfrak{M}.
\end{equation*}

 It is easy to find that Theorem~\ref{Hua's-theorem-k=3} is a direct corollary of the following theorem.
\begin{theorem}\label{theorem-3-2}
  Under the condition of Theorem~\ref{Hua's-theorem-k=3}, we have
  \begin{eqnarray}\label{Theorem-2-(6)}
    &   &   \int_{\frac{1}{\tau}}^{1+\frac{1}{\tau}}T_{1,1}(N,\alpha) T_{1,2}(N,\alpha)T_{3}(N,\alpha)e\big(-N\alpha\big)\mathrm{d}\alpha
                  \nonumber   \\
  &  =  &   \int_{\frac{1}{\tau}}^{1+\frac{1}{\tau}}S^2_1(N,\alpha)S_3(N,\alpha)e\big(-N\alpha\big)\mathrm{d}\alpha+O\big(N^{4/3}\mathcal{L}^{-B}\big),
  \end{eqnarray}
  where $B>0$ is arbitrary.
\end{theorem}

\subsection{Proof of Theorem~\ref{theorem-3-2}}

  First, take $\gamma=\gamma_3,\,\delta=\delta_3$ in Proposition~\ref{proposition-T3-S3}. By Lemma~\ref{T_1-square-mean-value} and Cauchy's inequality, we have
\begin{eqnarray} \label{Theorem-2-(7)}
  &     &   \int_{\frac{1}{\tau}}^{1+\frac{1}{\tau}}T_{1,1}(N,\alpha) T_{1,2}(N,\alpha)
            \Big(T_{3}(N,\alpha)-S_3(N,\alpha)\Big)e\big(-N\alpha\big)\mathrm{d}\alpha
                  \nonumber   \\
  & \ll &   \max_{\alpha\in[\frac{1}{\tau},1+\frac{1}{\tau}]}  \big|T_{3}(N,\alpha)-S_3(N,\alpha) \big|
            \times \int_{0}^1 \big|T_{1,1}(N,\alpha) T_{1,2}(N,\alpha)\big| \mathrm{d}\alpha
                  \nonumber   \\
  & \ll &  N^{(1-\delta_3)/3}\bigg(\int_0^1\big|T_{1,1}(N,\alpha)\big|^2\mathrm{d}\alpha\bigg)^{\frac{1}{2}}
            \bigg(\int_0^1\big|T_{1,2}(N,\alpha)\big|^2\mathrm{d}\alpha\bigg)^{\frac{1}{2}}
                  \nonumber   \\
  & \ll &  N^{(1-\delta_3)/3}\cdot N^{1-\gamma_1/2} \cdot N^{1-\gamma_2/2}
                  \nonumber   \\
  & \ll &   N^{7/3-(\gamma_1+\gamma_2)/2-\delta_3/3}.
\end{eqnarray}

Second, we divide the integral into the major arcs and the minor arcs.
\begin{eqnarray*}
  &    &  \int_{\frac{1}{\tau}}^{1+\frac{1}{\tau}}T_{1,1}(N,\alpha) T_{1,2}(N,\alpha)
               S_3(N,\alpha) e\big(-N\alpha\big)\mathrm{d}\alpha     \\
  & = &  \bigg\{  \int_{\mathfrak{M}}+\int_{\mathfrak{m}}   \bigg\}
            T_{1,1}(N,\alpha) T_{1,2}(N,\alpha)  S_3(N,\alpha) e\big(-N\alpha\big)\mathrm{d}\alpha .
\end{eqnarray*}
For $\alpha\in\mathfrak{m}$, there exist integers $a$ and $q$, such that (\ref{Diophantine-explicit}) holds with
\begin{equation*}
    \mathcal{Q}<q\leqslant\tau, \quad (a,q)=1,\quad 1\leqslant a\leqslant q.
\end{equation*}
Thus, we have
\begin{eqnarray*}
   \sup_{\alpha\in\mathfrak{m}}\big|S_{3}(N,\alpha)\big|
     &  \ll  &  N^{\frac{1}{3}+\varepsilon} \bigg(\frac{1}{Q}+\frac{1}{N^{1/6}}+\frac{\tau}{N}\bigg)^{\frac{1}{16}}  \\
     &  \ll  &  N^{\frac{1}{3}+\varepsilon} \big(N^{-\sigma}+N^{-\frac{1}{6}}\big)^{\frac{1}{16}}   \\
     &  \ll &   N^{\frac{1}{3}-\frac{\sigma}{16}+\varepsilon}.
\end{eqnarray*}
Using Lemma~\ref{T_1-square-mean-value} and Cauchy's inequality, we have
\begin{eqnarray}\label{Theorem-2-(8)}
   &     &      \int_{\mathfrak{m}} T_{1,1}(N,\alpha) T_{1,2}(N,\alpha)  S_3(N,\alpha) e\big(-N\alpha\big)\mathrm{d}\alpha
                      \nonumber   \\
   & \ll &      \sup_{\alpha\in\mathfrak{m}}\big|S_{3}(N,\alpha)\big|\times  \int_{\mathfrak{m}}
                \big| T_{1,1}(N,\alpha) T_{1,2}(N,\alpha)\big|\mathrm{d}\alpha
                      \nonumber   \\
   & \ll &      \sup_{\alpha\in\mathfrak{m}}  \big|S_{3}(N,\alpha)\big|\times
                \bigg(\int_0^1\big|T_{1,1}(N,\alpha)\big|^2\mathrm{d}\alpha\bigg)^{\frac{1}{2}}
                \bigg(\int_0^1\big|T_{1,2}(N,\alpha)\big|^2\mathrm{d}\alpha\bigg)^{\frac{1}{2}}
                      \nonumber   \\
   & \ll &      N^{\frac{1}{3}-\frac{\sigma}{16}+\varepsilon}\cdot N^{1-\gamma_1/2} \cdot N^{1-\gamma_2/2}
                      \nonumber   \\
   & \ll &      N^{\frac{7}{3}-\left(\frac{\sigma}{16}+\frac{\gamma_1+\gamma_2}{2}\right)+\varepsilon}.
\end{eqnarray}
For $\alpha\in\mathfrak{M}$, taking $\gamma=\gamma_i\,(i=1,2)$ in Lemma~\ref{S-1-N-alpha}, we have
\begin{eqnarray}\label{Theorem-2-(9)}
  &    &   \int_{\mathfrak{M}} T_{1,1}(N,\alpha) T_{1,2}(N,\alpha)  S_3(N,\alpha) e(-N\alpha)\mathrm{d}\alpha
                    \nonumber   \\
 & \ll & \int_{\mathfrak{M}}S_{1}^2(N,\alpha)S_{3}(N,\alpha)e(-N\alpha)\mathrm{d}\alpha
         + \int_{\mathfrak{M}}S_{1}(N,\alpha)S_{3}(N,\alpha)\cdot O(N^{1-\delta_1})\cdot e(-N\alpha)\mathrm{d}\alpha
                    \nonumber   \\
 &    &  + \int_{\mathfrak{M}} \big(O(N^{1-\delta_1})\big)^2S_{3}(N,\alpha)e(-N\alpha)\mathrm{d}\alpha
                    \nonumber   \\
 & =: &  \mathrm{I}+\mathrm{II}+\mathrm{III},
\end{eqnarray}
say. For $\mathrm{I}$, noting that
\begin{eqnarray*}
  &      &  \int_{\mathfrak{m}}  S_{1}^2(N,\alpha)S_{3}(N,\alpha)e(-N\alpha)\mathrm{d}\alpha
                   \nonumber   \\
  &  \ll &  \sup_{\alpha\in\mathfrak{m}} \big|S_3(N,\alpha)\big|\times \int_{0}^1\big|S_1(N,\alpha)\big|^2\mathrm{d}\alpha
                   \nonumber   \\
  & \ll  &  N^{\frac{1}{3}-\frac{\sigma}{16}+\varepsilon} \cdot N\ll N^{\frac{4}{3}-\frac{\sigma}{16}+\varepsilon},
\end{eqnarray*}
we have
\begin{eqnarray}\label{Theorem-2-(10)}
  \mathrm{I} & = & \int_{\frac{1}{\tau}}^{1+\frac{1}{\tau}}S^2_1(N,\alpha)S_3(N,\alpha)e\big(-N\alpha\big)\mathrm{d}\alpha
                    -\int_{\mathfrak{m}}  S_{1}^2(N,\alpha)S_{3}(N,\alpha)e(-N\alpha)\mathrm{d}\alpha
                        \nonumber   \\
  &  =  &  \int_{\frac{1}{\tau}}^{1+\frac{1}{\tau}}S^2_1(N,\alpha)S_3(N,\alpha)e\big(-N\alpha\big)\mathrm{d}\alpha
           +O\big(N^{\frac{4}{3}-\frac{\sigma}{16}+\varepsilon}\big).
\end{eqnarray}
It is easy to see that the measure of $\mathfrak{M}$ is
\begin{equation*}
  \ll \sum_{q\leqslant \mathcal{Q}}\sum_{a=1}^q\frac{1}{q\tau}\ll\frac{\mathcal{Q}}{\tau}\ll N^{2\sigma-1}.
\end{equation*}
Combining the trivial bound $S_1(N,\alpha)\ll N,\,S_3(N,\alpha)\ll N^{1/3}$, we have
\begin{eqnarray}\label{Theorem-2-(11)}
  \mathrm{II} & \ll &  N^{1-\delta_1}   \int_{\mathfrak{M}}  \big|S_1(N,\alpha) S_3(N,\alpha)\big|\mathrm{d}\alpha
                    \nonumber  \\
              & \ll &  N^{1-\delta_1} \cdot N\cdot N^{\frac{1}{3}}\cdot \mathrm{meas}(\mathfrak{M})
              \ll  N^{\frac{4}{3}+2\sigma-\delta_1}
\end{eqnarray}
and
\begin{eqnarray}\label{Theorem-2-(12)}
  \mathrm{III} & \ll &  N^{2-2\delta_1}   \int_{\mathfrak{M}}  \big| S_3(N,\alpha)\big|\mathrm{d}\alpha   \nonumber  \\
               & \ll &  N^{2-2\delta_1} \cdot N^{\frac{1}{3}} \cdot \mathrm{meas}(\mathfrak{M}) \ll N^{\frac{4}{3}+2(\sigma-\delta_1)}.
\end{eqnarray}
Collecting the above formulas (\ref{Theorem-2-(7)})-(\ref{Theorem-2-(12)}), under the conditions
\begin{equation*}
   \frac{\gamma_1+\gamma_2}{2}+\frac{\delta_3}{3}>1,\qquad \frac{\sigma}{16}+\frac{\gamma_1+\gamma_2}{2}>1,\qquad 2\sigma-\delta_1<0,
\end{equation*}
\begin{equation*}
   73(1-\gamma_i)+86\delta_1<9\quad (i=1,2),\qquad 1714(1-\gamma_3)+1725\delta_3<46,
\end{equation*}
i.e.
\begin{equation*}
   \frac{\gamma_1+\gamma_2}{2}+\frac{\delta_3}{3}>1,\qquad \frac{\gamma_1+\gamma_2}{2}+\frac{\delta_1}{32}>1,
\end{equation*}
\begin{equation*}
   73(1-\gamma_i)+86\delta_1<9\quad (i=1,2),\qquad 1714(1-\gamma_3)+1725\delta_3<46,
\end{equation*}
the equation (\ref{Theorem-2-(6)}) holds.

\bigskip
\bigskip

\textbf{Acknowledgement}

   The authors would like to express the most and the greatest sincere gratitude to Professor Wenguang Zhai for his valuable
advice and constant encouragement.

\end{document}